\documentclass[12pt,leqno,a4paper]{amsart}
\usepackage{amsmath,amsthm,amscd,amssymb,amsfonts,pdfsync, amsbsy}
\usepackage{latexsym}
\usepackage{exscale}

\usepackage[pagebackref]{hyperref}
\usepackage{color}

\definecolor{blue}{rgb}{0,0,1}
\hypersetup{colorlinks, linkcolor={blue},citecolor={red}, pagebackref}

\usepackage[utf8]{inputenc}

\calclayout

\allowdisplaybreaks

\DeclareMathOperator{\supp}{supp}
\DeclareMathOperator{\dist}{dist}
\DeclareMathOperator{\diam}{diam}
\DeclareMathOperator{\interior}{int}

\let\div\relax
\DeclareMathOperator{\div}{div}

\newcommand{\vertiii}[1]{{\left\vert\kern-0.25ex\left\vert\kern-0.25ex\left\vert #1
		\right\vert\kern-0.25ex\right\vert\kern-0.25ex\right\vert}}

\def\Xint#1{\mathchoice
{\XXint\displaystyle\textstyle{#1}}%
{\XXint\textstyle\scriptstyle{#1}}%
{\XXint\scriptstyle\scriptscriptstyle{#1}}%
{\XXint\scriptscriptstyle\scriptscriptstyle{#1}}%
\!\int}
\def\XXint#1#2#3{{\setbox0=\hbox{$#1{#2#3}{\int}$ }
\vcenter{\hbox{$#2#3$ }}\kern-.585\wd0}}
\def\barint{\Xint-}

\newcommand{\bariint}{\barint\mkern-11.5mu\barint}
\newcommand{\re}{\mathbb{R}}
\newcommand{\rest}[1]{\,\rule[-6pt]{.38pt}{12pt}_{\,#1}}

\renewcommand{\iint}{\int\mkern-13.5mu\int}
\renewcommand{\emptyset}{\mbox{\textup{\O}}}

\theoremstyle{plain}
\newtheorem{theorem}[equation]{Theorem}
\newtheorem{lemma}[equation]{Lemma}
\newtheorem{corollary}[equation]{Corollary}
\newtheorem{proposition}[equation]{Proposition}

\theoremstyle{definition}
\newtheorem{definition}[equation]{Definition}

\theoremstyle{remark}

\numberwithin{equation}{section}

\begin{document}

\title[Perturbations of elliptic operators in 1-sided chord-arc domains]{Perturbations of elliptic operators in 1-sided chord-arc domains. Part II: Non-symmetric operators and Carleson measure estimates}

\author{Juan Cavero}

\address{Juan Cavero
\\
Instituto de Ciencias Matem\'{a}ticas CSIC-UAM-UC3M-UCM
\\
Consejo Superior de Investigaciones Cient\'{\i}ficas
\\
C/ Nicol\'{a}s Cabrera, 13-15
\\
E-28049 Madrid, Spain} \email{juan.cavero@icmat.es}

\author{Steve Hofmann}

\address{Steve Hofmann
\\
Department of Mathematics
\\
University of Missouri
\\
Columbia, MO 65211, USA} \email{hofmanns@missouri.edu}

\author{Jos\'{e} Mar\'{\i}a Martell}

\address{Jos\'{e} Mar\'{\i}a Martell
\\
Instituto de Ciencias Matem\'{a}ticas CSIC-UAM-UC3M-UCM
\\
Consejo Superior de Investigaciones Cient\'{\i}ficas
\\
C/ Nicol\'{a}s Cabrera, 13-15
\\
E-28049 Madrid, Spain} \email{chema.martell@icmat.es}

\author{Tatiana Toro}

\address{Tatiana Toro
\\
University of Washington
\\
Department of Mathematics
\\
Seattle, WA 98195-4350, USA} \email{toro@uw.edu}

\thanks{The first author was partially supported by ``la Caixa''-Severo Ochoa international PhD Programme.
The first and third authors acknowledge financial
support from the Spanish Ministry of Economy and Competitiveness,
through the “Severo Ochoa” Programme for Centres of Excellence in
R\&D” (SEV-2015-0554). They also acknowledge that the research
leading to these results has received funding from the European
Research Council under the European Union's Seventh Framework
Programme (FP7/2007-2013)/ ERC agreement no. 615112 HAPDEGMT.
The second author was supported by NSF grant DMS-1664047.
The fourth author was partially supported by the Craig McKibben \& Sarah Merner Professor in Mathematics and by NSF grant DMS-1664867.}

\date{\today}

\subjclass[2010]{31B05, 35J08, 35J25, 42B99, 42B25, 42B37}

\keywords{Elliptic measure, Poisson kernel, Carleson measures, $A_\infty$ Muckenhoupt weights}

\begin{abstract}
We generalize to the setting of 1-sided chord-arc domains, that is, to domains satisfying the interior Corkscrew and Harnack Chain conditions (these are respectively scale-invariant/quantitative versions of the openness and path-connectedness) and which have an Ahlfors regular boundary, a result of Kenig-Kirchheim-Pipher-Toro, in which Carleson measure estimates for bounded solutions of the equation $Lu=-\div(A\nabla u) = 0$ with $A$ being a real (not necessarily symmetric) uniformly  elliptic matrix, imply that the corresponding elliptic measure belongs to the Muckenhoupt $A_\infty$ class with respect to
surface measure on the boundary. We present two applications of this result. In the first one we extend a perturbation result recently proved by Cavero-Hofmann-Martell presenting a simpler proof and allowing non-symmetric coefficients. Second, we prove that if an operator $L$ as above has locally Lipschitz  coefficients satisfying certain Carleson measure condition then $\omega_L\in A_\infty$ if and only if $\omega_{L^\top}\in A_\infty$. As a consequence, we can remove one of the main assumptions in the non-symmetric case of a result of Hofmann-Martell-Toro and show that if the coefficients satisfy a slightly stronger Carleson measure condition the membership of the elliptic measure associated with $L$ to the class $A_\infty$ yields that the domain is indeed a chord-arc domain.
\end{abstract}

\maketitle

\setcounter{tocdepth}{2}
\tableofcontents

\section{Introduction and Main results}

F. and M. Riesz showed in \cite{RR} that harmonic measure is absolutely continuous with respect to the surface measure for any simply connected domain in the complex plane whose boundary is rectifiable. Since then, one can find many references in the literature studying how the previous result, or its quantitative version obtained by Lavrentiev \cite{Lav}, can be extended to higher dimensions. In doing that, some kind of ``strong'' connectivity hypotheses is needed (as shown by the counter example in \cite{BJ}). Dahlberg in \cite{MR0466593} established that harmonic measure satisfies a quantitative version of absolute continuity with respect to the surface measure for every Lipschitz domain. That quantitative version says that harmonic measure is in the Muckenhoupt class of weights $A_\infty$, and more precisely it belongs to $RH_2$, the class of weights satisfying a reverse Hölder condition with exponent $2$. 

Jerison and Kenig \cite{MR676988} introduced a new class of domains called NTA (non-tangentially accessible). These domains satisfy interior and exterior Corkscrew conditions (these are quantitative versions of the fact that the domain and its exterior are open sets). They also satisfy an interior Harnack Chain condition (which is a quantitative version of the path-connectivity). In this class of domains they developed the boundary regularity theory for harmonic functions, they also established the properties of the harmonic measure, and the Green function. NTA domains whose boundary is Ahlfors regular are called of type chord-arc. In this class of domains which include Lipschitz domains David-Jerison \cite{MR1078740} and independently Semmes  \cite{Sem} proved that the harmonic measure is an $A_\infty$ weight with respect to surface measure to the boundary.  It belongs to some class $RH_p$ with $p>1$.

Recently a big effort has been made to understand in what domains and for what operators the elliptic measure is an $A_\infty$ weight with respect to surface measure to the boundary
of the domain. One context where the theory has been satisfactorily developed is that of $1$-sided chord-arc domains. These are open sets $\Omega\subset \mathbb{R}^{n+1}$, $n\ge 2$, whose boundaries $\partial\Omega$ are $n$-dimensional Ahlfors regular (cf. Definition \ref{defAR}), and which satisfy interior (but not exterior) Corkscrew and Harnack Chain conditions  see Definitions \ref{defCS} and \ref{defHC}
below). In \cite{hofmartell,  hofmartelltuero} the authors show that in the setting of 1-sided chord-arc domains, harmonic measure is in $A_\infty(\partial\Omega)$ (cf. \ref{defi:Ainfty}) if and only if $\partial\Omega$ is uniformly rectifiable (a quantitative version of rectifiability). It was shown later in \cite{MR3626548} that under the same background hypothesis, if $\partial\Omega$ is uniformly rectifiable then $\Omega$ satisfies an exterior corkscrew condition and hence $\Omega$ is a chord-arc domain. All these together and, additionally, \cite{MR3626548} in conjunction with
\cite{MR1078740} or \cite{Sem}, give a characterization of chord-arc domains, or a characterization of the uniform rectifiability of the boundary,  in terms of the membership of harmonic measure to the class $A_\infty(\partial\Omega)$. For other elliptic operators $Lu=-\div(A\nabla u)$ with variable coefficients it was shown recently in \cite{HMT-var} that the same characterization holds provided $A$ is locally Lipschitz and has appropriately controlled oscillation near the boundary.

This paper is the second part of a series of two articles where we consider perturbation of real elliptic operators in the setting of 1-sided chord-arc domains. In the first paper of the series \cite{cavhofmartell} we worked with symmetric operators and studied perturbations that preserve the $A_\infty(\partial\Omega)$ property extending the work of \cite{MR1114608, MR3107693, MPT2} (see also \cite{MR1828387}, \cite{MR2833577, MR2655385}) to the setting of 1-sided chord-arc domains. It was shown that if the disagreement between two elliptic symmetric matrices satisfies certain Carleson measure condition, then one of the associated elliptic measures is in $A_\infty(\partial\Omega)$ if and only if the other one is in $A_\infty(\partial\Omega)$. In other words, the property that the elliptic measure belongs to $A_\infty(\partial\Omega)$ is stable under Carleson  measure type perturbations. That result was proved using the so-called extrapolation of Carleson measures,  which originated in  \cite{MR1020043} (see also \cite{MR1828387, MR1879847, MR1934198}), in the form developed in \cite{MR2833577, MR2655385} (see also \cite{hofmartell}).  The method is a bootstrapping argument, based on
the Corona construction of Carleson \cite{Car} and Carleson and Garnett \cite{CG}, that, roughly speaking, allows one to reduce matters to the case in which the perturbation is small in some sawtooth subdomains.  Implicit in the proof of the perturbation result in \cite{cavhofmartell} one can find the treatment of the case in which the perturbation is small, and this allowed the authors to obtain that  for sufficiently small perturbations, not only the class $A_\infty$ is preserved but one can also keep the same exponent in the corresponding reverse Hölder class.

In the present paper we work in the same setting of 1-sided chord-arc domains and consider real not necessarily symmetric elliptic operators. Our first goal is to establish that for any real elliptic operator non-necessarily symmetric $L$, the property that all bounded solutions of $L$ satisfy Carleson measure estimates yields  $\omega_L\in A_\infty(\partial\Omega)$. This extends the work \cite{KKiPT} where they treated bounded Lipschitz domains and domains above the graph of a Lipschitz function.
That the converse is true (hence both properties are equivalent) follows from \cite{HMT-general}  where a more general estimate is obtained. Indeed, assuming that  $\omega_L\in A_\infty(\partial\Omega)$ then it is shown that the conical square function is controlled by the non-tangential maximal function in every $L^p(\partial\Omega)$ for every $1<p<\infty$ where both are applied to solutions of $L$. Applying this estimate with $p=2$ to a bounded solution 
one obtains the desired Carleson. Here, nevertheless, we present a simpler and novel argument for the latter fact. The precise result is as follows:

\begin{theorem}\label{theor:cme-implies-ainf}
	Let $\Omega\subset\re^{n+1}$ be a 1-sided $\mathrm{CAD}$ and let $Lu=-\div(A\nabla u)$ be a real (not necessarily symmetric) elliptic operator (cf. Definition \ref{ellipticoperator}). The following statements are equivalent:
	\begin{list}{$(\theenumi)$}{\usecounter{enumi}\leftmargin=1cm \labelwidth=1cm \itemsep=0.1cm \topsep=.2cm \renewcommand{\theenumi}{\alph{enumi}}}
	
\item	Every bounded weak solution of $Lu=0$ satisfies a Carleson measure estimate, that is, there exists $C$ such that every $u\in W^{1,2}_{\rm loc}(\Omega)\cap L^\infty(\Omega)$ with $Lu=0$  in $\Omega$ in the weak sense,  satisfies the Carleson measure condition
	\begin{equation}\label{cmeestimate}
	\sup_{\substack{x\in\partial\Omega \\ 0<r<\infty}}\frac{1}{r^n}\iint_{B(x,r)\cap\Omega}|\nabla u(X)|^2\delta(X)\,dX\leq C\|u\|_{L^\infty(\Omega)}^2.
	\end{equation}

\item $\omega_L\in A_\infty(\partial\Omega)$  (cf. Definition \ref{defi:Ainfty}).
	\end{list}
\end{theorem}
\bigskip

Our second goal is to use the previous characterization to extend the ``large'' constant  perturbation result from \cite{cavhofmartell} to the non-symmetric case:

\begin{theorem}\label{theor:perturbation-ainf-cme}
	Let $\Omega\subset\re^{n+1}$, $n\ge 2$, be a 1-sided $\mathrm{CAD}$ (cf. Definition \ref{def-1CAD}). Let $L_1u=-\div(A_1\nabla u)$ and $L_0u=-\div(A_0\nabla u)$ be real (not necessarily symmetric) elliptic operators (cf. Definition \ref{ellipticoperator}). Define the disagreement between $A_1$ and $A_0$ in $\Omega$ by
	\begin{equation}\label{discrepancia}
	\varrho(A_1, A_0)(X):=\sup_{Y\in B(X,\delta(X)/2)}|A_1(Y)-A_0(Y)|,\qquad X\in\Omega,
	\end{equation}
	where $\delta(X):=\dist(X,\partial\Omega)$, and assume that it satisfies the Carleson measure condition
	\begin{equation}\label{eq:defi-vertiii}
	\sup_{\substack{x\in\partial\Omega \\ 0<r<\diam(\partial\Omega)}}\frac{1}{\sigma(B(x,r)\cap\partial\Omega)}\iint_{B(x,r)\cap\Omega}\frac{\varrho(A_1, A_0)(X)^2}{\delta(X)}\,dX<\infty.
	\end{equation}	
	Then,  $\omega_{L_0}\in A_\infty(\partial\Omega)$ if and only if $\omega_{L_1}\in A_\infty(\partial\Omega)$ (cf. Definition \ref{defi:Ainfty}).  	 	
\end{theorem}
\medskip

To prove this result we use a novel approach which is  interesting on its own right and is conceptually simpler. The bottom line is that assuming that $\omega_{L_0}\in A_\infty(\partial\Omega)$ and based on Theorem \ref{theor:cme-implies-ainf} we just need to establish that all bounded solutions for $L_1$ satisfy the aforementioned Carleson measure estimates, rather than trying to establish the ``more delicate'' condition  $\omega_{L_1}\in A_\infty(\partial\Omega)$. In doing this we exploit the fact that $\omega_{L_0}\in A_\infty(\Omega)$ to  find a sawtooth domain whose boundary has with ample contact with $\partial\Omega$, where the averages of $\omega_{L_0}$ are essentially constant. Hence
in \eqref{cmeestimate} one can replace $\delta$ by $G_{L_0}$ in a sawtooth with ample contact. This in turn allows us to perform some integrations by parts to conclude the desired estimate. We would like to emphasize that this approach cannot be used to get the ``small'' constant perturbation since
that requires to directly show that the two elliptic measures are in the same reverse Hölder class without passing through the Carleson measure estimates.

\medskip

Our last main result establishes a connection between the elliptic measures of an operator and its adjoint
assuming that the derivative of the antisymmetric part of the matrix defining the operator satisfies
some Carleson measure condition:

\begin{theorem}\label{theor:perturbation-ainf-cme-A-At}
Let $\Omega\subset\re^{n+1}$, $n\ge 2$, be a 1-sided $\mathrm{CAD}$ (cf. Definition \ref{def-1CAD}). Let $Lu=-\div(A\nabla u)$ be  a real (not necessarily symmetric) elliptic operator (cf. Definition \ref{ellipticoperator}), let $L^\top$ denote the transpose of $L$ (i.e, $L^\top u=-\div(A^\top\nabla u)$ with $A^\top$ being the transpose matrix of $A$), and let $L^{\rm sym}=\frac{L+L^\top}{2}$ be the symmetric part of $L$. Assume that $(A-A^\top)\in {\rm Lip}_{\rm loc}(\Omega)$ and let
\begin{equation}\label{def-B0}
\div_C (A-A^\top)(X)
=
\bigg(\sum_{i=1}^{n+1}\partial_{i}(a_{i,j}-a_{j,i})(X)\bigg)_{1\leq j\leq n+1},\qquad X\in\Omega.
\end{equation}
Assume that  the following Carleson measure estimate holds
\begin{equation}\label{carleson-B0}
\sup_{\substack{x\in\partial\Omega \\ 0<r<\diam(\partial\Omega)}} \frac{1}{\sigma(B(x,r)\cap\partial\Omega)}
\iint_{B(x,r)\cap\Omega}\big|\div_C (A-A^\top)(X)\big|^2\delta(X)\,dX<\infty.
\end{equation}
Then $\omega_{L}\in A_\infty(\partial\Omega)$ if and only if $\omega_{L^\top}\in A_\infty(\partial\Omega)$ if and only if $\omega_{L^{\rm sym}}\in A_\infty(\partial\Omega)$ (cf. Definition \ref{defi:Ainfty}). 
\end{theorem}
\medskip

As an immediate consequence of the previous result we obtain the following:

\begin{corollary}\label{cor-tt-1}
Let $\Omega\subset\re^{n+1}$, $n\ge 2$, be a 1-sided $\mathrm{CAD}$ (cf. Definition \ref{def-1CAD}). Let $Lu=-\div(A\nabla u)$ be  a real (not necessarily symmetric) elliptic operator (cf. Definition \ref{ellipticoperator}). Assume that $A\in {\rm Lip}_{\rm loc}(\Omega)$, $|\nabla A|\,\delta\in L^{\infty}(\Omega)$ and the following Carleson measure estimate
\begin{equation}
\sup_{\substack{x\in\partial\Omega \\ 0<r<\diam(\partial\Omega)}} \frac{1}{\sigma(B(x,r)\cap\partial\Omega)}
\iint_{B(x,r)\cap\Omega}\big|\nabla A(X)\big|^2\delta(X)\,dX<\infty.
\end{equation}
Then $\omega_{L}\in A_\infty(\partial\Omega)$ if and only if $\omega_{L^\top}\in A_\infty(\partial\Omega)$.

In particular, if one further assumes that
\begin{equation}
\sup_{\substack{x\in\partial\Omega \\ 0<r<\diam(\partial\Omega)}} \frac{1}{\sigma(B(x,r)\cap\partial\Omega)}
\iint_{B(x,r)\cap\Omega}\big|\nabla A(X)\big|\,dX<\infty,
\end{equation}
then
\begin{equation}
\omega_L\in A_\infty(\partial\Omega)
\qquad\Longrightarrow\qquad
\Omega\mbox{ is a $\mathrm{CAD}$ (cf. Definition \ref{def-1CAD})}.
\end{equation}
\end{corollary}
\medskip

The first part of Corollary \ref{cor-tt-1} follows from Theorem \ref{theor:perturbation-ainf-cme-A-At}. For the second part, we notice that
once  $\omega_L\in A_\infty(\partial\Omega)$ implies, after using the first part, that  $\omega_{L^\top}\in A_\infty(\partial\Omega)$. In turn, we can then invoke \cite[Theorem 1.5] {HMT-var} to conclude that $\Omega$ is a CAD. Note that comparing this with \cite[Theorem 1.5] {HMT-var} what we are proving is that with the given background hypotheses one just needs to assume $\omega_L\in A_\infty(\partial\Omega)$, and the assumption $\omega_{L^\top}\in A_\infty(\partial\Omega)$ is redundant.

\medskip

The organization of the paper is as follows. In Section \ref{section:Prelim} we present some of the needed preliminaries, notations, definitions and some of the PDE estimates which will be needed throughout the paper. Section \ref{section:Proof-CME} contains the proof of Theorem \ref{theor:cme-implies-ainf}. Theorems \ref{theor:perturbation-ainf-cme} and \ref{theor:perturbation-ainf-cme-A-At} are proved in Section \ref{section:Proof-Perturb}, as a matter of facts both results are  particular cases of the much more general Theorem \ref{theor:perturbation-general}.

\section{Preliminaries}\label{section:Prelim}
\subsection{Notation and conventions}
\begin{list}{$\bullet$}{\leftmargin=0.4cm  \itemsep=0.2cm}

\item  Our ambient space is $\re^{n+1}$, $n\ge 2$.

\item We use the letters $c$, $C$ to denote harmless positive constants, not necessarily the same at each occurrence, which depend only on dimension and the constants appearing in the hypotheses of the theorems (which we refer to as the ``allowable parameters''). We shall also sometimes write $a\lesssim b$ and $a\approx b$ to mean, respectively, that $a\leq C b$ and $0<c\leq a/b\leq C$, where the constants $c$ and $C$ are as above, unless explicitly noted to the contrary. Moreover, if $c$ and $C$ depend on some given parameter $\eta$, which is somehow relevant, we write $a\lesssim_\eta b$ and $a\approx_\eta b$. At times, we shall designate by $M$ a particular constant whose value will remain unchanged throughout the proof of a given lemma or proposition, but which may have a different value during the proof of a different lemma or proposition.

\item Given a domain (i.e., open and connected) $\Omega\subset\re^{n+1}$, we shall use lower case letters $x,y,z$, etc., to denote points on $\partial\Omega$, and capital letters $X,Y,Z$, etc., to denote generic points in $\re^{n+1}$ (especially those in $\Omega$).


\item The open $(n+1)$-dimensional Euclidean ball of radius $r$ will be denoted $B(x,r)$ when the center $x$ lies on $\partial\Omega$, or $B(X,r)$ when the center $X\in\re^{n+1}\setminus \partial\Omega$. A ``surface ball'' is denoted $\Delta(x,r):=B(x,r)\cap \partial\Omega$, and unless otherwise specified it is implicitly assumed that $x\in\partial\Omega$. Also if $\partial\Omega$ is bounded, we typically assume that $0<r\lesssim\diam(\partial\Omega)$, so that $\Delta=\partial\Omega$ if $\diam(\partial\Omega)<r\lesssim\diam(\partial\Omega)$.

\item Given a Euclidean ball $B$ or surface ball $\Delta$, its radius will be denoted $r(B)$ or $r(\Delta)$
respectively.

\item Given a Euclidean ball $B=B(X,r)$ or surface ball $\Delta=\Delta(x,r)$, its concentric dilate by a factor of $\kappa>0$ will be denoted by $\kappa B=B(X,\kappa r)$ or $\kappa\Delta=\Delta(x,\kappa r)$.

\item For $X\in\re^{n+1}$, we set $\delta_{\partial\Omega}(X):=\dist(X,\partial\Omega)$. Sometimes, when clear from the context we will omit the subscript $\partial\Omega$ and simply write $\delta(X)$.

\item We let $H^n$ denote the $n$-dimensional Hausdorff measure, and let $\sigma_{\partial\Omega}:=H^n\rest{\partial\Omega}$ denote the ``surface measure'' on $\partial\Omega$. For a closed set $E\subset\re^{n+1}$ we will use the notation $\sigma_{E}:=H^n\rest{E}$. When clear from the context we will also omit the subscript and simply write $\sigma$.

\item For a Borel set $A\subset\re^{n+1}$, we let $\mathbf{1}_A$ denote the usual indicator function of $A$, i.e., $\mathbf{1}_A(x)=1$ if $x\in A$, and $\mathbf{1}_A(x)=0$ if $x\notin A$.

\item For a Borel set $A\subset\re^{n+1}$, we let $\interior(A)$ denote the interior of $A$, and $\overline{A}$ denote the closure of $A$. If $A\subset \partial\Omega$, $\interior(A)$ will denote the relative interior, i.e., the largest relatively open set in $\partial\Omega$ contained in $A$. Thus, for $A\subset \partial\Omega$, the boundary is then well defined by $\partial A:=\overline{A}\setminus\interior(A)$.
\item For a Borel set $A\subset\re^{n+1}$, we denote by $C(A)$ the space of continuous functions on $A$ and by $C_c(A)$ the subspace of $C(A)$ with compact support in $A$. Note that if $A$ is compact then $C(A)\equiv C_c(A)$.

\item For a Borel set $A\subset \partial\Omega$ with $0<\sigma(A)<\infty$, we write $\barint_{A}f\,d\sigma:=\sigma(A)^{-1}\int_A f\,d\sigma$.

\item We shall use the letter $I$ (and sometimes $J$) to denote a closed $(n+1)$-dimensional Euclidean cube with sides parallel to the co-ordinate axes, and we let $\ell(I)$ denote the side length of $I$. We use $Q$ to denote a dyadic ``cube'' on $E\subset\re^{n+1}$. The latter exists, given that $E$ is $\mathrm{AR}$ (cf. \cite{davidsemmes1}, \cite{christ}), and enjoy certain properties which we enumerate in Lemma \ref{dyadiccubes} below.
\end{list}
\medskip

\subsection{Some definitions}
\begin{definition}[\textbf{Corkscrew condition}]\label{defCS}
Following \cite{MR676988}, we say that an open set $\Omega\subset\re^{n+1}$ satisfies the ``Corkscrew condition'' if for some uniform constant $c\in(0,1)$ and for every surface ball $\Delta:=\Delta(x,r)=B(x,r)\cap\partial\Omega$ with $x\in\partial\Omega$ and $0<r<\diam(\partial\Omega)$, there is a ball $B(X_{\Delta},cr)\subset B(x,r)\cap\Omega$. The point $X_\Delta\in\Omega$ is called a ``corkscrew point'' relative to $\Delta$. Note that we may allow $r<C\diam(\partial\Omega)$ for any fixed $C$, simply by adjusting the constant $c$.
\end{definition}

\begin{definition}[\textbf{Harnack Chain condition}]\label{defHC}
Again following \cite{MR676988}, we say that $\Omega\subset\re^{n+1}$ satisfies the Harnack Chain condition if there is a uniform constant $C$ such that for every $\rho>0$, $\Theta\geq 1$, and every pair of points $X,X'\in\Omega$ with $\delta(X),\delta(X')\geq\rho$ and $|X-X'|<\Theta\rho$, there is a chain of open balls $B_1,\dots,B_N\subset\Omega$, $N\leq C(\Theta)$, with $X\in B_1$, $X'\in B_N$, $B_k\cap B_{k+1}\neq\emptyset$ and $C^{-1}\diam(B_k)\leq\dist(B_k,\partial\Omega)\leq C\diam(B_k)$. The chain of balls is called a ``Harnack Chain''.                                                                               \end{definition}

\begin{definition}[\textbf{Ahlfors regular}]\label{defAR}
We say that a closed set $E\subset\re^{n+1}$ is $n$-dimensional $\mathrm{AR}$ (or simply $\mathrm{AR}$), if there is some uniform constant $C=C_{\mathrm{AR}}$ such that
$$
C^{-1}r^n\leq H^n(E\cap B(x,r))\leq C r^n,\quad 0<r< \diam(E),\quad x\in E.
$$
\end{definition}

\begin{definition}[\textbf{1-sided chord-arc domain} and \textbf{chord-arc domain}]\label{def-1CAD}
We say that $\Omega\subset\re^{n+1}$ is a ``1-sided chord-arc domain'' (1-sided CAD for short) if it satisfies the Corkscrew and Harnack Chain conditions and if $\partial\Omega$ is $\mathrm{AR}$.  Analogously, we say that $\Omega\subset\re^{n+1}$ is a ``chord-arc domain'' (CAD for short) if it is a 1-sided CAD and additionally $\Omega_{\rm ext}=\re^{n+1}\setminus \overline{\Omega}$ also satisfies the Corkscrew condition.
\end{definition}
\medskip

\subsection{Dyadic grids and sawtooths}\label{sec2.3}

We give a lemma concerning the existence of a ``dyadic grid'':

\begin{lemma}[\textbf{``Dyadic grid''} {\cite{davidsemmes1, davidsemmes2}, \cite{christ}}]\label{dyadiccubes}
Suppose that $E\subset\re^{n+1}$ is $n$-di\-men\-sional $\mathrm{AR}$. Then there exist constants $a_0>0$, $\eta>0$ and $C<\infty$ depending only on dimension and the $\mathrm{AR}$ constant, such that for each $k\in\mathbb{Z}$ there is a collection of Borel sets (``cubes'')
$$
\mathbb{D}_k:=\big\{Q_j^k\subset \partial\Omega:\:j\in\mathcal{J}_k\big\},
$$
where $\mathcal{J}_k$ denotes some (possibly finite) index set depending on $k$, satisfying:
\begin{list}{$(\theenumi)$}{\usecounter{enumi}\leftmargin=1cm \labelwidth=1cm \itemsep=0.1cm \topsep=.2cm \renewcommand{\theenumi}{\alph{enumi}}}
\item $E=\bigcup_jQ_j^k$ for each $k\in\mathbb{Z}$.
\item If $m\geq k$ then either $Q_i^m\subset Q_j^k$ or $Q_i^m\cap Q_j^k=\emptyset$.
\item For each $j,k\in\mathbb{Z}$ and each $m>k$, there is a unique $i\in\mathbb{Z}$ such that $Q_j^k\subset Q_i^m$.
\item $\diam(Q_j^k)\leq C\,2^{-k}$.
\item Each $Q_j^k$ contains some ``surface ball'' $\Delta(x_j^k,a_02^{-k})=B(x_j^k,a_02^{-k})\cap E$.
\item $H^n\big(\big\{x\in Q_j^k:\,\dist(x,E\setminus Q_j^k)\leq\tau 2^{-k}\big\}\big)\leq C\tau^\eta H^n(Q_j^k)$, for all $j,k\in\mathbb{Z}$ and for all $\tau\in(0,a_0)$.
\end{list}
\end{lemma}

A few remarks are in order concerning this lemma.
\begin{list}{$\bullet$}{\leftmargin=0.4cm  \itemsep=0.2cm}

\item In the setting of a general space of homogeneous type, this lemma has been proved by Christ
\cite{christ}, with the dyadic parameter $1/2$ replaced by some constant $\delta \in (0,1)$.
In fact, one may always take $\delta = 1/2$ (cf.  \cite[Proof of Proposition 2.12]{HMMM}).
In the presence of the Ahlfors regularity property, the result already appears in \cite{davidsemmes1, davidsemmes2}.

\item We shall denote by $\mathbb{D}(E)$ the collection of all relevant $Q_j^k$, i.e.,
$$
\mathbb{D}(E):=\bigcup_k\mathbb{D}_k,
$$
where, if $\diam(E)$ is finite, the union runs over those $k\in\mathbb{Z}$ such that $2^{-k}\lesssim\diam(E)$.
\item For a dyadic cube $Q\in\mathbb{D}_k$, we shall set $\ell(Q)=2^{-k}$, and we shall refer to this quantity as the ``length'' of $Q$. It is clear that $\ell(Q)\approx\diam(Q)$. Also, for $Q\in\mathbb{D}(E)$ we will set $k(Q)=k$ if $Q\in\mathbb{D}_k$.

\item Properties $(d)$ and $(e)$ imply that for each cube $Q\in\mathbb{D}(E)$, there is a point $x_Q\in E$, a Euclidean ball $B(x_Q,r_Q)$ and a surface ball $\Delta(x_Q,r_Q):=B(x_Q,r_Q)\cap E$ such that $c\ell(Q)\leq r_Q\leq\ell(Q)$, for some uniform constant $c>0$, and
    \begin{equation}\label{deltaQ}
    \Delta(x_Q,2r_Q)\subset Q\subset\Delta(x_Q,Cr_Q)
    \end{equation}
    for some uniform constant $C>1$. We shall denote these balls and surface balls by
    \begin{equation}\label{deltaQ2}
    B_Q:=B(x_Q,r_Q),\qquad\Delta_Q:=\Delta(x_Q,r_Q),
    \end{equation}
    \begin{equation}\label{deltaQ3}
    \widetilde{B}_Q:=B(x_Q,Cr_Q),\qquad\widetilde{\Delta}_Q:=\Delta(x_Q,Cr_Q),
    \end{equation}
    and we shall refer to the point $x_Q$ as the ``center'' of $Q$.

\item Let $\Omega\subset\re^{n+1}$ be an open set satisfying the Corkscrew condition and such that $\partial\Omega$ is $\mathrm{AR}$. Given $Q\in\mathbb{D}(\partial\Omega)$ we define the ``corkscrew point relative to $Q$'' as $X_Q:=X_{\Delta_Q}$. We note that
    $$
    \delta(X_Q)\approx\dist(X_Q,Q)\approx\diam(Q).
    $$
\end{list}
\medskip

Following \cite[Section 3]{hofmartell} we next introduce the notion of ``Carleson region'' and ``discretized sawtooth''. Given a cube  $Q\in\mathbb{D}(E)$, the ``discretized Carleson region'' $\mathbb{D}_Q$ relative to $Q$ is defined by
$$
\mathbb{D}_Q:=\big\{Q'\in\mathbb{D}(E):\,Q'\subset Q\big\}.
$$
Let $\mathcal{F}=\{Q_i\}\subset\mathbb{D}(E)$ be a family of disjoint cubes. The ``global discretized sawtooth'' relative to $\mathcal{F}$ is the collection of cubes $Q\in\mathbb{D}(E)$ that are not contained in any $Q_i\in\mathcal{F}$, that is,
$$
\mathbb{D}_\mathcal{F}:=\mathbb{D}(E)\setminus\bigcup_{Q_i\in\mathcal{F}}\mathbb{D}_{Q_i}.
$$
For a given $Q\in\mathbb{D}(E)$, the ``local discretized sawtooth'' relative to $\mathcal{F}$ is the collection of cubes in $\mathbb{D}_Q$ that are not contained in any $Q_i\in\mathcal{F}$ or, equivalently,
$$
\mathbb{D}_{\mathcal{F},Q}:=\mathbb{D}_{Q}\setminus\bigcup_{Q_i\in\mathcal{F}}\mathbb{D}_{Q_i}=\mathbb{D}_\mathcal{F}\cap\mathbb{D}_Q.
$$

We also introduce the ``geometric'' Carleson regions and sawtooths. In the sequel, $\Omega\subset\re^{n+1}$ ($n\geq 2$) will be a 1-sided $\mathrm{CAD}$. Given $Q\in\mathbb{D}(\partial\Omega)$ we want to define some associated regions which inherit the good properties of $\Omega$. Let $\mathcal{W}=\mathcal{W}(\Omega)$ denote a collection of (closed) dyadic Whitney cubes of $\Omega\subset\re^{n+1}$, so that the cubes in $\mathcal{W}$ form a pairwise non-overlapping covering of $\Omega$, which satisfy
\begin{equation}\label{constwhitney}
4\diam(I)\leq\dist(4I,\partial\Omega)\leq\dist(I,\partial\Omega)\leq 40\diam(I),\qquad\forall I\in\mathcal{W},
\end{equation}
and
$$
\diam(I_1)\approx\diam(I_2),\,\text{ whenever }I_1\text{ and }I_2\text{ touch}.
$$
Let $X(I)$ denote the center of $I$, let $\ell(I)$ denote the sidelength of $I$, and write $k=k_I$ if $\ell(I)=2^{-k}$.

Given $0<\lambda<1$ and $I\in\mathcal{W}$ we write $I^*=(1+\lambda)I$ for the ``fattening'' of $I$. By taking $\lambda$ small enough, we can arrange matters, so that, first, $\dist(I^*,J^*)\approx\dist(I,J)$ for every $I,J\in\mathcal{W}$, and secondly, $I^*$ meets $J^*$ if and only if $\partial I$ meets $\partial J$ (the fattening thus ensures overlap of $I^*$ and $J^*$ for any pair $I,J\in\mathcal{W}$ whose boundaries touch, so that the Harnack Chain property then holds locally in $I^*\cup J^*$, with constants depending upon $\lambda$). By picking $\lambda$ sufficiently small, say $0<\lambda<\lambda_0$, we may also suppose that there is $\tau\in(1/2,1)$ such that for distinct $I,J\in\mathcal{W}$, we have that $\tau J\cap I^*=\emptyset$. In what follows we will need to work with dilations $I^{**}=(1+2\lambda)I$ or $I^{***}=(1+4\lambda)I$, and in order to ensure that the same properties hold we further assume that $0<\lambda<\lambda_0/4$.

For every $Q\in\mathbb{D}(\partial\Omega)$ we can construct a family $\mathcal{W}_Q^*\subset\mathcal{W}$, and define
$$
U_Q:=\bigcup_{I\in\mathcal{W}_Q^*}I^*,
$$
satisfying the following properties: $X_Q\in U_Q$ (actually, $X_Q$ can be taken to be the center of some Whitney cube $I\in\mathcal{W}_Q^*$), and there are uniform constants $k^*$ and $K_0$ such that
\begin{gather*}
k(Q)-k^*\leq k_I\leq k(Q)+k^*,\quad\forall I\in\mathcal{W}_Q^*,
\\[4pt]
X(I)\rightarrow_{U_Q} X_Q,\quad\forall I\in\mathcal{W}_Q^*,
\\[4pt]
\dist(I,Q)\leq K_0 2^{-k(Q)},\quad\forall I\in\mathcal{W}_Q^*.
\end{gather*}
Here, $X(I)\rightarrow_{U_Q} X_Q$ means that the interior of $U_Q$ contains all balls in a Harnack Chain (in $\Omega$) connecting $X(I)$ to $X_Q$, and moreover, for any point $Z$ contained in any ball in the Harnack Chain, we have $\dist(Z,\partial\Omega)\approx\dist(Z,\Omega\setminus U_Q)$ with uniform control of the implicit constants. The constants $k^*, K_0$ and the implicit constants in the condition $X(I)\rightarrow_{U_Q} X_Q$, depend on at most allowable parameters and on $\lambda$. Moreover, given $I\in\mathcal{W}$ we have that $I\in\mathcal{W}_{Q_I}^*$, where $Q_I\in\mathbb{D}(\partial\Omega)$ satisfies $\ell(Q_I)=\ell(I)$, and contains any fixed $\widehat{y}\in\partial\Omega$ such that $\dist(I,\partial\Omega)=\dist(I,\widehat{y})$. The reader is referred to \cite{hofmartell} for full details.

For a given $Q\in\mathbb{D}(\partial\Omega)$, the ``Carleson box'' relative to $Q$ is defined by
$$
T_Q:=\interior\bigg(\bigcup_{Q'\in\mathbb{D}_Q}U_{Q'}\bigg).
$$
For a given family $\mathcal{F}=\{Q_i\}$ of pairwise disjoint cubes and a given $Q\in\mathbb{D}(\partial\Omega)$, we define the ``local sawtooth region'' relative to $\mathcal{F}$ by
\begin{equation}\label{defomegafq}
\Omega_{\mathcal{F},Q}=\interior\bigg(\bigcup_{Q'\in\mathbb{D}_{\mathcal{F},Q}}U_{Q'}\bigg)=\interior\bigg(\bigcup_{I\in\mathcal{W}_{\mathcal{F},Q}}I^*\bigg),
\end{equation}
where $\mathcal{W}_{\mathcal{F},Q}:=\bigcup_{Q'\in\mathbb{D}_{\mathcal{F},Q}}\mathcal{W}_{Q'}^*$. Analogously, we can slightly fatten the Whitney boxes and use $I^{**}$ to define new fattened Whitney regions and sawtooth domains. More precisely, for every $Q\in\mathbb{D}(\partial\Omega)$,
$$
T_Q^*:=\interior\bigg(\bigcup_{Q'\in\mathbb{D}_Q}U_{Q'}^*\bigg),\qquad\Omega^*_{\mathcal{F},Q}:=\interior\bigg(\bigcup_{Q'\in\mathbb{D}_Q}U_{Q'}^*\bigg), \qquad U_{Q}^*:=\bigcup_{I\in\mathcal{W}_Q^*}I^{**}.
$$
Similarly, we can define $T_Q^{**}$, $\Omega^{**}_{\mathcal{F},Q}$ and $U^{**}_{Q}$ by using $I^{***}$ in place of $I^{**}$.

Given a pairwise disjoint family $\mathcal{F}\subset\mathbb{D}$ (we also allow $\mathcal{F}$ to be the null set) and a constant $\rho>0$, we derive another family $\mathcal{F}(\rho)\subset\mathbb{D}$ from $\mathcal{F}$ as follows. Augment $\mathcal{F}$ by adding cubes $Q\in\mathbb{D}$ whose sidelength $\ell(Q)\leq\rho$ and let $\mathcal{F}(\rho)$ denote the corresponding collection of maximal cubes. Note that the corresponding discrete sawtooth region $\mathbb{D}_{\mathcal{F}(\rho)}$ is the union of all cubes $Q\in\mathbb{D}_{\mathcal{F}}$ such that $\ell(Q)>\rho$. For a given constant $\rho$ and a cube $Q\in\mathbb{D}$, let $\mathbb{D}_{\mathcal{F}(\rho),Q}$ denote the local discrete sawtooth region and let $\Omega_{\mathcal{F}(\rho),Q}$ denote the geometric sawtooth region relative to it.

Given $Q\in\mathbb{D}(\partial\Omega)$ and $0<\varepsilon<1$, if we take $\mathcal{F}_0=\emptyset$, one has that $\mathcal{F}_0(\varepsilon\ell(Q))$ is the collection
of $Q'\in\mathbb{D}(\partial\Omega)$ such that $\varepsilon\ell(Q)/2 < \ell(Q') \leq\varepsilon\ell(Q)$, hence $\mathbb{D}_{\mathcal{F}_0(\varepsilon\ell(Q)),Q}=\{Q'\in\mathbb{D}_Q:\,\ell(Q')>\varepsilon\ell(Q)\}$. We then introduce $U_{Q,\varepsilon}=\Omega_{\mathcal{F}_0(\varepsilon\ell(Q)),Q}$ , which is a Whitney region relative to $Q$ whose distance to $\partial\Omega$ is of the order of $\varepsilon\ell(Q)$. For later use, we observe that given $Q_0\in\mathbb{D}(\partial\Omega)$, the sets $\{U_{Q,\varepsilon}\}_{Q\in\mathbb{D}_{Q_0}}$ have bounded overlap with constant that may
depend on $\varepsilon$. Indeed, suppose that there is $X\in U_{Q,\varepsilon}\cap U_{Q',\varepsilon}$ with $Q,Q'\in\mathbb{D}_{Q_0}$. By construction
$\ell(Q)\approx_{\varepsilon}\delta(X)\approx_{\varepsilon}\ell(Q')$ and $\dist(Q,Q')\leq\dist(X,Q)+\dist(X,Q')\lesssim_{\varepsilon}\ell(Q)+\ell(Q')\approx_{\varepsilon}\ell(Q)$.
The bounded overlap property, with constants depending on $\varepsilon$, follows then at once.

\medskip

Following \cite{hofmartell}, one can easily see that there exist constants $0<\kappa_1<1$ and $\kappa_0\geq \max\{2C,4/c\}$ (with $C$ the constant in \eqref{deltaQ3}, and $c$ such that $c\ell(Q)\leq r_Q$), depending only on the allowable parameters, so that
\begin{equation}\label{definicionkappa12}
\kappa_1B_Q\cap\Omega\subset T_Q\subset T_Q^*\subset T_Q^{**}\subset \overline{T_Q^{**}}\subset\kappa_0B_Q\cap\overline{\Omega}=:\tfrac{1}{2}B_Q^*\cap\overline{\Omega},
\end{equation}
where $B_Q$ is defined as in \eqref{deltaQ2}.
\medskip

\subsection{PDE estimates}
Next, we recall several facts concerning the elliptic measures and the Green functions. For our first results we will only assume that $\Omega\subset\re^{n+1}$, $n\geq 2$, is an open set, not necessarily connected, with $\partial\Omega$ satisfying the $\mathrm{AR}$ property. Later we will focus on the case where $\Omega$ is a 1-sided $\mathrm{CAD}$.

\begin{definition}\label{ellipticoperator}
Let $Lu=-\div(A\nabla u)$ be a variable coefficient second order divergence form operator with $A(X)=(a_{i,j}(X))_{i,j=1}^{n+1}$ being a real (not necessarily symmetric) matrix with $a_{i,j}\in L^{\infty}(\Omega)$ for $1\leq i,j\leq n+1$, and $A$ uniformly elliptic, that is, there exists $\Lambda\geq 1$ such that
$$
\Lambda^{-1}|\xi|^2\leq A(X)\xi\cdot\xi,\qquad |A(X)\xi\cdot\zeta|\leq\Lambda|\xi||\zeta|,
$$
for all $\xi,\zeta\in\re^{n+1}$ and almost every $X\in\Omega$.
\end{definition}

In what follows we will only be working with this kind of operators, we will refer to them as ``elliptic operators'' for the sake of simplicity. We write $L^\top$ to denote the transpose of $L$, or, in other words, $L^\top u=-\div(A^\top\nabla u)$ with $A^\top$ being the transpose matrix of $A$.

We say that a function $u\in W^{1,2}_{\text{loc}}(\Omega)$ is a weak solution of $Lu=0$ in $\Omega$, or that $Lu=0$ in the weak sense, if
$$
\iint_{\Omega}A(X)\nabla u(X)\cdot\nabla\varphi(X)\,dX=0,\qquad\forall\varphi\in C_c^{\infty}(\Omega).
$$

Associated with $L$ and $L^\top$ one can respectively construct the elliptic measures $\{\omega_L^X\}_{X\in\Omega}$ and $\{\omega_{L^\top}^X\}_{X\in\Omega}$, and the Green functions $G_L$ and $G_{L^\top}$ (see \cite{HMT-general} for full details). We next present some definitions and properties that will be used throughout this paper.

\begin{definition}\label{defi:Ainfty}
Let $\Omega\subset\re^{n+1}$ be a 1-sided $\mathrm{CAD}$ and let $L$ be a real (non-necessarily symmetric) elliptic operator. We say that the elliptic measure $\omega_{L}\in A_\infty(\partial\Omega)$ if there exist constants $0<\alpha,\beta<1$ such that given an arbitrary surface ball $\Delta_0=B_0\cap  \partial\Omega$, with $B_0=B(x_0,r_0)$, $x_0\in  \partial\Omega$, $0<r_0<\diam( \partial\Omega)$, and for every surface ball $\Delta=B\cap \partial\Omega$ centered at $\partial\Omega$ with $B\subset B_0$, and for every Borel set $F\subset\Delta$, we have that
\begin{equation}\label{Ainfty-alpha-beta}
\frac{\omega_L^{X_{\Delta_0}}(F)}{\omega_L^{X_{\Delta_0}}(\Delta)}\le \alpha
\implies
\frac{\sigma(F)}{\sigma(\Delta)}\le \beta.
\end{equation}
\end{definition}

It is well known (see \cite{MR807149}, \cite{coifman1974}) that since $\sigma$ is a doubling measure (recall that $\partial\Omega$ satisfies the $\mathrm{AR}$ condition), $\omega_L\in A_\infty(\partial\Omega)$ if and only if $\omega_L\ll\sigma$ in $\partial\Omega$  and there exists $1<q<\infty$ such that
for every $\Delta_0$ and $\Delta$ as above
$$
\bigg(\barint_{\Delta} k_L^{X_{\Delta_0}}(x)^{q}\,d\sigma(x)\bigg)^{\frac{1}{q}}\leq
C\barint_{\Delta}k_L^{X_{\Delta_0}}(x)\,d\sigma(x),
$$
where $k_L^{X_{\Delta_0}}=d\omega_L^{X_{\Delta_0}}/d\sigma$ is the Radon-Nikodym derivative.  Moreover since $\Omega$ is a 1-sided $\mathrm{CAD}$ the latter is equivalent to the scale invariant estimate (see \cite{HMT-general})
\begin{equation}\label{rhqL}
\int_{\Delta_0}k_{L}^{X_{\Delta_0}}(y)^q\,d\sigma(y)\leq C\sigma(\Delta_0)^{1-q}.
\end{equation}
for every surface ball $\Delta_0$.

\begin{lemma}\label{bourgain}
Suppose that $\Omega\subset\re^{n+1}$ is an open set such that $\partial\Omega$ satisfies the $\mathrm{AR}$ property. Let $L$ be an elliptic operator, there exist constants $c<1$ and $C>1$ (depending only on the $\mathrm{AR}$ constant and on the ellipticity of $L$) such that for every $x\in\partial\Omega$ and every $0<r<\diam(\partial\Omega)$, we have
$$
\omega_L^Y(\Delta(x,r))\geq \frac{1}{C},\qquad\forall\, Y\in B(x,cr)\cap\Omega.
$$
\end{lemma}

We refer the reader to \cite[Lemma 1]{bou} for the proof in the harmonic case and to \cite{HMT-general} for general elliptic operators. See also \cite[Theorem 6.18]{HKM} and \cite[Section 3]{ZHAO}.

\medskip

The proofs of the following lemmas may be found in \cite{HMT-general}. We note that, in particular, the $\mathrm{AR}$ hypothesis implies that $\partial\Omega$ satisfies the Capacity Density Condition, hence $\partial\Omega$ is Wiener regular at every point (see \cite[Lemma 3.27]{hoflemartellnystrom}).

\begin{lemma}\label{lemagreen}
Suppose that $\Omega\subset\re^{n+1}$ is an open set such that $\partial\Omega$ satisfies the $\mathrm{AR}$ property. Given an elliptic operator $L$, there exist $C>1$ (depending only on dimension and on the ellipticity of $L$) and $c_\theta>0$ (depending on the above parameters and on $\theta\in (0,1)$) such that $G_L$, the Green function associated with $L$, satisfies
\begin{gather}\label{sizestimate}
G_L(X,Y)\leq C|X-Y|^{1-n};
\\[0.15cm]
c_\theta|X-Y|^{1-n}\leq G_L(X,Y),\quad\text{if }\,|X-Y|\leq\theta\delta(X),\quad\theta\in(0,1);
\\[0.15cm]
G_L(\cdot,Y)\in C\big(\overline{\Omega}\setminus\{Y\}\big)\quad\text{and}\quad G_L(\cdot,Y)\rest{\partial\Omega}\equiv 0\quad\forall\,Y\in\Omega;
\\[0.15cm]
G_L(X,Y)\geq 0,\quad\forall X,Y\in\Omega,\quad X\neq Y;
\\[0.15cm]
G_L(X,Y)=G_{L^\top}(Y,X),\quad\forall X,Y\in\Omega,\quad X\neq Y.
\end{gather}
Moreover, $G_L(\cdot,Y)\in W^{1,2}_{\rm loc}(\Omega\setminus\{Y\})$ for every $Y\in \Omega$, and satisfies $LG_L(\cdot,Y)=\delta_Y$ in the weak sense in $\Omega$, that is,
\begin{equation}\label{greenpole}
\int_{\Omega}A(X)\nabla_XG_L(X,Y)\cdot\nabla\varphi(X)\,dX=\varphi(Y),\quad\forall\,\varphi\in C_c^{\infty}(\Omega).
\end{equation}
\end{lemma}
\medskip

\begin{lemma}\label{proppde}
Suppose that $\Omega\subset\re^{n+1}$ is a 1-sided $\mathrm{CAD}$. Let $L$ be an elliptic operator, there exist $C$, $0<\gamma\leq 1$ (depending only on dimension, the $1$-sided $\mathrm{CAD}$ constants and the ellipticity of $L$), such that for every $B_0=B(x_0,r_0)$ with $x_0\in\partial\Omega$, $0<r_0<\diam(\partial\Omega)$, and $\Delta_0=B_0\cap\partial\Omega$ we have the following properties:
\begin{list}{$(\theenumi)$}{\usecounter{enumi}\leftmargin=1cm \labelwidth=1cm \itemsep=0.1cm \topsep=.2cm \renewcommand{\theenumi}{\alph{enumi}}}

\item If $0\leq u\in W^{1,2}_{\rm loc}(B_0\cap\Omega)\cap C(\overline{B_0\cap\Omega})$ is a weak solution of $Lu=0$ in $B_0\cap\Omega$ such that $u\equiv 0$ in $\Delta_0$, then
$$
u(X)\leq C\bigg(\frac{|X-x_0|}{r_0}\bigg)^\gamma\sup_{Y\in\overline{B_0\cap\Omega}}u(Y),\qquad\forall\,X\in \frac{1}{2}B_0\cap\Omega.
$$

\item  If $B=B(x,r)$ with $x\in\partial\Omega$ and $\Delta=B\cap\partial\Omega$ is such that $2B\subset B_0$, then for all $X\in\Omega\setminus B_0$ we have that
    $$
\frac{1}{C}\omega_L^X(\Delta)\leq r^{n-1} G_L(X,X_\Delta)\leq C\omega_L^X(\Delta).
    $$

\item  If $X\in\Omega\setminus 4B_0$ then
    $$
    \omega_{L}^X(2\Delta_0)\leq C\omega_{L}^X(\Delta_0).
    $$

\item  If $B=B(x,r)$ with $x\in\partial\Omega$ and $\Delta:=B\cap\partial\Omega$ is such that $B\subset B_0$, then for every $X\in\Omega\setminus 2\kappa_0B_0$ with $\kappa_0$ as in \eqref{definicionkappa12}, we have that
$$
\frac{1}{C}\omega_L^{X_{\Delta_0}}(\Delta)\leq \frac{\omega_L^X(\Delta)}{\omega_L^X(\Delta_0)}\leq C\omega_L^{X_{\Delta_0}}(\Delta).
$$
Moreover, if $F\subset\Delta_0$ is a Borel set then
$$
\frac{1}{C}\omega_L^{X_{\Delta_0}}(F)\leq \frac{\omega_L^X(F)}{\omega_L^X(\Delta_0)}\leq C\omega_L^{X_{\Delta_0}}(F).
$$

\end{list}
\end{lemma}
\bigskip

\section{Proof of Theorem \ref{theor:cme-implies-ainf}}\label{section:Proof-CME}

\subsection{The Carleson measure condition implies $A_\infty$ }\label{section:Proof-CME:a->b}

To prove that : $(a)\Longrightarrow (b)$ we first introduce some notation.

\begin{definition}\label{defi-cover}
	Let $E\subset\mathbb{R}^{n+1}$ be an $n$-dimensional $\mathrm{AR}$ set. Fix $Q_0\in\mathbb{D}(E)$ and let $\mu$ be a regular Borel measure on $Q_0$. Given $\varepsilon_0\in(0,1)$ and a Borel set $F\subset Q_0$, a good $\varepsilon_0$-cover of $F$ with respect to $\mu$, of length $k\in\mathbb{N}$, is a collection $\{\mathcal{O}_\ell\}_{\ell=1}^k$ of Borel subsets of $Q_0$, together with pairwise disjoint families $\mathcal{F}_\ell=\{Q_i^\ell\}\subset\mathbb{D}_{Q_0}$, such that
	\begin{list}{$(\theenumi)$}{\usecounter{enumi}\leftmargin=1cm \labelwidth=1cm \itemsep=0.1cm \topsep=.2cm \renewcommand{\theenumi}{\alph{enumi}}}
		\item $F\subset \mathcal{O}_k\subset\mathcal{O}_{k-1}\subset\cdots\subset\mathcal{O}_2\subset\mathcal{O}_1\subset Q_0$,

		\item $\mathcal{O}_\ell=\bigcup_{Q_i^\ell\in\mathcal{F}_\ell}Q_i^\ell,\qquad 1\leq\ell\leq k$,

		\item $\mu(\mathcal{O}_\ell\cap Q_i^{\ell-1})\leq\varepsilon_0\,\mu(Q_i^{\ell-1}),\qquad\forall\,Q_i^{\ell-1}\in\mathcal{F}_{\ell-1},\quad2\leq\ell\leq k$.
	\end{list}

\end{definition}
\medskip

\begin{lemma}\label{lemma-eps-cover}
	If $\{\mathcal{O}_\ell\}_{\ell=1}^k$ is a good $\varepsilon_0$-cover of $F$ with respect to $\mu$ of length $k\in\mathbb{N}$ then
	\begin{equation}\label{iteration-cover}
	\mu(\mathcal{O}_\ell\cap Q_i^m)\leq\varepsilon_0^{\ell-m}\mu(Q_i^m),\qquad\forall\,Q_i^m\in\mathcal{F}_m,\qquad 1\leq m\leq \ell\leq k.
	\end{equation}
\end{lemma}

\begin{proof}	
	Fix $1\le\ell\le k$ and we proceed by induction in $m$. If $m=\ell$ the estimate is trivial since
	$\mu(\mathcal{O}_\ell\cap Q_i^\ell)=\mu(Q_i^\ell)$. If $m=\ell-1$ (in which case necessarily $\ell\ge 2$) then \eqref{iteration-cover} follows directly from $(c)$ in Definition \ref{defi-cover}.
	Assume next that \eqref{iteration-cover} holds for some fixed $2\le m\le \ell$ and we prove it for $m-1$ in place of $m$. We first claim that for every $Q_i^{m-1}\in \mathcal{F}_{m-1}$ there holds
	\begin{equation}\label{rq4r34f}
	\mathcal{O}_\ell\cap Q_i^{m-1}\subset
	\bigcup_{\substack{Q_j^{m}\in\mathcal{F}_{m}\\Q_j^{m}\subsetneq Q_i^{m-1}}}
	\mathcal{O}_\ell\cap Q_j^{m}.
	\end{equation}
	To see this, take $x\in\mathcal{O}_\ell\cap Q_i^{m-1}\subset \mathcal{O}_m$. Hence, there exists a unique $Q_j^{m}\in\mathcal{F}_{m}$ such that $x\in Q_j^{m}$ and consequently  either $Q_i^{m-1}\subset Q_j^{m}$ or $Q_j^{m}\subsetneq Q_i^{m-1}$.  If $Q_i^{m-1}\subset Q_j^{m}$ then $\mu(Q_i^{m-1})=\mu(\mathcal{O}_{m}\cap Q_i^{m-1})\leq\varepsilon_0\mu(Q_i^{m-1})$, by $(c)$ in Definition \ref{defi-cover}, and this is a contradiction since $0<\varepsilon_0<1$. Thus, $Q_j^{m}\subsetneq Q_i^{m-1}$ and \eqref{rq4r34f} holds. Therefore
	\begin{multline*}
	\mu(\mathcal{O}_\ell\cap Q_i^{m-1})
	\le
	\sum_{\substack{Q_j^{m}\in\mathcal{F}_{m}\\Q_j^{m}\subsetneq Q_i^{m-1}}}
	\mu(\mathcal{O}_\ell\cap Q_j^{m})
	\le
	\varepsilon_0^{\ell-m}
	\sum_{\substack{Q_j^{m}\in\mathcal{F}_{m}\\Q_j^{m}\subsetneq Q_i^{m-1}}}
	\mu(Q_j^m)
	\\
	\le
	\varepsilon_0^{\ell-m}
	\mu(\mathcal{O}_m\cap Q_i^{m-1})
	\le
	\varepsilon_0^{\ell-(m-1)}
	\mu(Q_i^{m-1}),
	\end{multline*}
	where we have applied the induction hypothesis to the $Q_j^{m}$'s and the properties of the good $\varepsilon_0$-cover.
\end{proof}

\begin{lemma}\label{extract-eps-cover}
	Let $E\subset\mathbb{R}^{n+1}$ be an $n$-dimensional $\mathrm{AR}$ set and fix $Q_0\in\mathbb{D}(E)$. Let $\mu$ be a regular Borel measure on $Q_0$ and assume that it is dyadically doubling on $Q_0$, that is,  there exists $C_\mu\ge 1$ such that $\mu(Q^*)\leq C_\mu\mu(Q)$ for every $Q\in\mathbb{D}_{Q_0}\setminus\{Q_0\}$, with $Q^*\supset Q$  and $\ell(Q^*)=2\ell(Q)$ (i.e., $Q^*$ is the ``dyadic parent'' of $Q$).
	For every $0<\varepsilon_0\le e^{-1}$, if $F\subset Q_0$ with $\mu(F)\leq\alpha\mu(Q_0)$ and $0<\alpha\leq\varepsilon_0^2/(2C_\mu^2)$ then $F$ has a good $\varepsilon_0$-cover with respect to $\mu$ of length $k_0=k_0(\alpha,\varepsilon_0)\in\mathbb{N}$, $k_0\geq 2$, which satisfies $
	k_0\,\approx\,\tfrac{\log{\alpha^{-1}}}{\log{\varepsilon_0^{-1}}}.
	$
	In particular, if $\mu(F)=0$, then $F$ has a good $\varepsilon_0$-cover of arbitrary length.
\end{lemma}

\begin{proof}
	Fix $\varepsilon_0$, $F$ and $\alpha$ as in the statement and write $a:=C_\mu/\varepsilon_0>1$. Note that since $0<\alpha<\varepsilon_0^2/(2C_\mu^2)=a^{-2}/2$
	there is a unique $k_0=k_0(\alpha,\varepsilon_0)\in\mathbb{N}$, $k_0\geq 2$, such that
	\[
	a^{-k_0-1}<\;2\alpha\;\leq\;a^{-k_0},
	\]
	and our choice of $\varepsilon_0$ gives that
	\begin{equation}\label{est-k0}
	\frac{1}{3(1+\log{C_\mu})}\frac{\log{\alpha^{-1}}}{\log{\varepsilon_0^{-1}}}\leq\, k_0\,\leq \frac{\log{\alpha^{-1}}}{\log{\varepsilon_0^{-1}}}.
	\end{equation}
	Since $\mu(F)\leq\alpha\mu(Q_0)$, by outer regularity there exists a relatively open set $U\subset E$ such that $F\subset U$ and  $\mu(U\setminus F)<\alpha\mu(Q_0)$. Set $\widetilde{F}:=U\cap Q_0\subset Q_0$ and 	define the level sets
	\[
	\Omega_k:=\big\{x\in Q_0:\:M_{\mu,Q_0}^d(\mathbf{1}_{\widetilde{F}})(x)>a^{-k} \big\},\qquad 1\leq k\leq k_0,
	\]
	where $M_{\mu,Q_0}^d$ is the local dyadic maximal operator with respect to $\mu$ given by
	$$
	M_{\mu,Q_0}^{d}f(x):=\sup_{x\in Q\in\mathbb{D}_{Q_0}}\frac{1}{\mu(Q)}\int_Q f(y)\,d\mu(y),\qquad f\in L^1_{\rm loc}(Q_0,d\mu).
	$$
	Clearly, $\Omega_1\subset\Omega_2\subset\dots\subset\Omega_{k_0}\subset Q_0$. Moreover, $\widetilde{F}\subset\Omega_1$. To see this fix $x\in \widetilde{F}$ and use that $U$ is relatively open to find $B_x=B(x,r_x)$ with $r_x>0$ so that $B_x\cap E\subset U$. Take next $Q_x\in\mathbb{D}$ with $Q_x\ni x$ so that $\ell(Q_x)
	<\ell(Q_0)$ and $\diam(Q_x)<r_x$. Since $x\in \widetilde{F}\cap Q_x\subset Q_x\cap Q_0$ and $\ell(Q_x)<\ell(Q_0)$ it follows that $Q_x\in\mathbb{D}_{Q_0}$. Also since
	$\diam(Q_x)<r_x$ we easily see that $Q_x\subset B_x\cap E\subset U$ and eventually we have obtained that $Q_x\subset \widetilde{F}$ which in turn gives
	\[
	M_{\mu,Q_0}^d(\mathbf{1}_{\widetilde{F}})(x)\geq\frac{\mu(\widetilde{F}\cap Q_x)}{\mu(Q_x)}\,=\,1\,>\, a^{-1}.
	\]
	Hence, $x\in\Omega_1$ as desired.
	
	All the previous observations show that $F\subset \widetilde{F}\subset \Omega_1\subset\Omega_2\subset\dots\subset\Omega_{k_0}\subset Q_0$ and in particular $\Omega_k\neq \emptyset$ for every $k\ge 1$. Moreover, by our choice of $k_0$, we have that for every $1\le k\le k_0$
	\[
	\mu(\widetilde{F})\leq\mu(U)\leq\mu(U\setminus F)+\mu(F)<2\alpha\mu(Q_0)\leq a^{-k_0}\mu(Q_0)\le a^{-k}\mu(Q_0).
	\]
	Subdividing $Q_0$ dyadically we can then select a pairwise disjoint collection of cubes $\mathcal{F}_{k}=\{Q_i^{k}\}\subset\mathbb{D}_{Q_0}\setminus\{Q_0\}$ which is maximal with respect to the property that
	\begin{equation}\label{propmax}
	\mu(\widetilde{F}\cap Q_i^k)>a^{-k}\mu(Q_i^k),
	\end{equation}
	and also $\Omega_k=\bigcup_{Q_i^k\in\mathcal{F}_k}Q_i^k$ (note that $\mathcal{F}_k\neq\emptyset$ since $\Omega_k\neq\emptyset$). By the maximality of $\mathcal{F}_k$ as well as the dyadic doubling property of $\mu$ we obtain that
	\begin{equation}\label{maximality}
	\frac{\mu(\widetilde{F}\cap Q_i^k)}{\mu(Q_i^k)}\leq C_\mu\,\frac{\mu(\widetilde{F}\cap (Q_i^k)^*)}{\mu((Q_i^k)^*)}\leq C_\mu\,a^{-k},
	\end{equation}
	where $(Q_i^k)^*$ is the dyadic parent of $Q_i^k$.

	Next we claim that for each $Q_j^{k+1}\in\mathcal{F}_{k+1}$ we have that $\mu(\Omega_k\cap Q_j^{k+1})\leq\varepsilon_0\mu(Q_j^{k+1})$. To see this we first observe that if $Q_i^k\cap Q_j^{k+1}\neq\emptyset$,  then necessarily  $Q_i^k\subset Q_j^{k+1}$, for  otherwise $Q_j^{k+1}\subsetneq Q_i^k$ and by the maximality of $\mathcal{F}_{k+1}$ using \eqref{propmax} we would have that $a^{-k}\mu(Q_i^k)<\mu(\widetilde{F}\cap Q_i^k)\leq a^{-k-1}\mu(Q_i^k)$, which leads to a contradiction since $a>1$. Hence, $Q_i^k\subset Q_j^{k+1}$ whenever $Q_i^k\cap Q_j^{k+1}\neq\emptyset$. Using this, \eqref{propmax}, and \eqref{maximality} (for $Q_j^{k+1}$ and $k+1$ replacing $Q_i^k$ and $k$ respectively), we have that
	\begin{multline*}
	\mu(\Omega_k\cap Q_j^{k+1})
	=
	\sum_{Q_i^k:\,Q_i^k\subset Q_j^{k+1}}\mu(Q_i^k\cap Q_j^{k+1})
	=
	\sum_{Q_i^k:\,Q_i^k\subset Q_j^{k+1}}\mu(Q_i^k)
	\\
	< a^k\sum_{Q_i^k:\,Q_i^k\subset Q_j^{k+1}}\mu(\widetilde{F}\cap Q_i^k)\leq a^k\,\mu(\widetilde{F}\cap Q_j^{k+1})
	\leq a^{-1}\,C_\mu\,\mu(Q_j^{k+1})=\varepsilon_0\,\mu(Q_j^{k+1}),
	\end{multline*}
	and this proves the claim.
	
	To complete the proof of the lemma we define $\mathcal{O}_k:=\Omega_{k_0-k+1}$ and note that the sets $\{\mathcal{O}_k\}_{k=1}^{k_0}$ form a good $\varepsilon_0$-cover of $F$, with respect to $\mu$, of length $k_0$ which satisfies \eqref{est-k0}. Finally we observe that if $\mu(F)=0$, then $\alpha$ can be taken arbitrarily small, hence $k_0$, the length of the good $\varepsilon_0$-cover of $F$, can be taken as large as desired by \eqref{est-k0}.
\end{proof}
\bigskip

Given  $Q_0\in\mathbb{D}(\partial\Omega)$ and for every $\eta\in(0,1)$ we define the modified non-tangential cone
\begin{equation}\label{def-whitney-eta}
\Gamma_{Q_0}^{\eta}(x):=\bigcup_{\shortstack{$\scriptstyle Q\in\mathbb{D}_{Q_0}$\\$\scriptstyle Q\ni x$}}U_{Q,\eta^3},\qquad U_{Q,\eta^3}=\bigcup_{\shortstack{$\scriptstyle Q'\in\mathbb{D}_{Q}$\\$\scriptstyle \ell(Q')>\eta^3\ell(Q)$}}U_{Q'}.
\end{equation}
As already noted in Section 2, the sets $\{U_{Q,\eta^3}\}_{Q\in\mathbb{D}_{Q_0}}$ have bounded overlap with constant depending on $\eta$.

\begin{lemma}\label{sq-function->M}
	There exist $0<\eta\ll 1$, depending  only on dimension, the $1$-sided $\mathrm{CAD}$ constants and the ellipticity of $L$, and
	$\alpha_0\in (0,1)$, $C_\eta\ge 1$ both depending on the same parameters and additionally on $\eta$, such that for every $Q_0\in\mathbb{D}$, for every $0<\alpha<\alpha_0$, and for every Borel set $F\subset Q_0$ satisfying $\omega_L^{X_{Q_0}}(F)\le\alpha \omega_L^{X_{Q_0}}(Q_0)$, there exists a Borel set $S\subset Q_0$ such that the bounded weak solution $u(X)=\omega^X_L(S)$ satisfies
	\begin{equation}\label{def-mod-sq-function}
	S_{Q_0}^{\eta}u(x):=\bigg(\iint_{\Gamma_{Q_0}^{\eta}(x)}|\nabla u(Y)|^2\delta(Y)^{1-n}\,dY\bigg)^{1/2}
	\ge
	C_\eta^{-1} \big(\log{\alpha^{-1}}\big)^{\frac12}
	,\qquad \forall\,x\in F,
	\end{equation}
\end{lemma}

Assuming this result momentarily, we can now prove Theorem \ref{theor:cme-implies-ainf}.

\begin{proof}[Proof of Proof of Theorem \ref{theor:cme-implies-ainf}: $(a)\Longrightarrow (b)$]
	Our first goal is to see that given $\beta\in (0,1)$ there exists $\alpha\in (0,1)$ so that for every $Q_0\in\mathbb{D}$ and
	every Borel set $F\subset Q_0$, we have that
	\begin{equation}\label{goal:Ainfty}
	\frac{\omega_L^{X_{Q_0}}(F)}{\omega_L^{X_{Q_0}}(Q_0)}\le \alpha
	\qquad\implies\qquad
	\frac{\sigma(F)}{\sigma(Q_0)}\le\beta.
	\end{equation}
	Fix $\beta\in (0,1)$ and $Q_0\in\mathbb{D}$, and take a Borel set $F\subset Q_0$ so that
	${\omega_L^{X_{Q_0}}(F)}\le \alpha\omega_L^{X_{Q_0}}(Q_0) $ where $\alpha\in (0,1)$ is to be chosen. Applying Lemma \ref{sq-function->M}, if we assume that $0<\alpha<\alpha_0$, then  $u(X)=\omega^X_L(S)$ satisfies \eqref{def-mod-sq-function} and therefore
	\begin{multline}\label{endproof1}
	C_\eta^{-2} \log{\alpha^{-1}}\sigma(F)
	\leq
	\int_F S_{Q_0}^\eta u(x)^2\,d\sigma(x)
	\\
	\leq
	\int_{Q_0}\bigg(\iint_{\Gamma_{Q_0}^\eta(x)}|\nabla u(Y)|^2\delta(Y)^{1-n}\,dY\bigg)\,d\sigma(x)
	\\
	=\iint_{B_{Q_0}^*\cap\Omega}|\nabla u(Y)|^2\delta(Y)^{1-n}\bigg(\int_{Q_0}\mathbf{1}_{\Gamma_{Q_0}^\eta(x)}(Y)\,d\sigma(x)\bigg)\,dY
	\end{multline}
	where we have used that $\Gamma_{Q_0}^\eta(x)\subset T_{Q_0}\subset B_{Q_0}^*\cap\Omega$ (see \eqref{definicionkappa12}), and Fubini's theorem. To estimate the inner integral we fix $Y\in B_{Q_0}^*\cap\Omega$ and $\widehat{y}\in\mathbb{D}(\partial\Omega)$ such that $|Y-\widehat{y}|=\delta(Y)$. We claim that
	\begin{equation}\label{claimendproof}
	\big\{x\in Q_0:\:Y\in\Gamma_{Q_0}^\eta(x)\big\}\subset\Delta(\widehat{y},C\eta^{-3}\delta(Y)).
	\end{equation}
	To show this let $x\in Q_0$ be such that $Y\in\Gamma_{Q_0}^\eta(x)$. Then there exists $Q\in\mathbb{D}_{Q_0}$ such that $x\in Q$ and $Y\in U_{Q,\eta^3}$. Hence, there is $Q'\in\mathbb{D}_Q$ with $\ell(Q')>\eta^3\ell(Q)$ such that $Y\in U_{Q'}$ and consequently $\delta(Y)\approx \dist(Y,Q')\approx \ell(Q')$. Then,
	$$
	|x-\widehat{y}|\leq\diam(Q)+\dist(Y,Q')+\delta(Y)
	\lesssim\ell(Q)+\delta(Y)\leq C\eta^{-3}\delta(Y),
	$$
	thus $x\in\Delta(\widehat{y},C\eta^{-3}\delta(Y))$ as desired. If we now use \eqref{claimendproof} and the $\mathrm{AR}$ property  we conclude that for every $Y\in B_{Q_0}^*\cap\Omega$
	\[
	\int_{Q_0}\mathbf{1}_{\Gamma_{Q_0}^\eta(x)}(Y)\,d\sigma(x)
	\le
	\sigma(\Delta(\widehat{y},C\eta^{-3}\delta(Y)))
	\lesssim
	\eta^{-3n}\delta(Y)^n.
	\]
	Plugging this into \eqref{endproof1} and using \eqref{cmeestimate}, since  $u\in W^{1,2}_{\rm loc}(\Omega)\cap L^\infty(\Omega)$ with $Lu=0$ in the weak sense in $\Omega$,
	we obtain
	\[
	C_\eta^{-2} \log{\alpha^{-1}}\sigma(F)
	\lesssim\eta^{-3n}\iint_{B_{Q_0}^*\cap\Omega}|\nabla u(Y)|^2\delta(Y)\,dY
	\lesssim
	\eta^{-3n}\sigma(\Delta_{Q_0}^*)
	\leq C\eta^{-3n}\sigma(Q_0),
	\]
	where we have used that $\Delta_{Q_0}^*=B_{Q_0}^*\cap\partial\Omega$, that $0\le u(X)\le \omega^X(\partial\Omega)\le 1$ and  that $\partial \Omega$ is AR.  Rearranging the terms we see that $\sigma(F)/\sigma(Q_0)\le \beta$ provided $0<\alpha<\min\{\alpha_0,e^{-C C_\eta^2\eta^{-3n}\beta^{-1}}\}$ and \eqref{goal:Ainfty} follows.

	Next we see that \eqref{goal:Ainfty} implies that $\omega_{L}\in A_\infty(\partial\Omega)$. To see this we first obtain a dyadic-$A_\infty$ condition. Fix $Q^0, Q_0\in\mathbb{D}$ with $Q_0\subset Q^0$. Lemma \ref{proppde} parts $(c)$ and $(d)$, Harnack's inequality and Lemma \ref{bourgain} gives for every $F\subset Q_0$
	\begin{equation}\label{cop-dyadic}
	\frac{1}{C_1}\frac{\omega_L^{X_{Q_0}}(F)}{\omega_L^{X_{Q_0}}(Q_0)}
	\leq
	\frac{\omega_L^{X_{Q^0}}(F)}{\omega_L^{X_{Q^0}}(Q_0)}
	\leq
	C_1\frac{\omega_L^{X_{Q_0}}(F)}{\omega_L^{X_{Q_0}}(Q_0)}.
	\end{equation}

	With all these in hand we fix $\beta\in (0,1)$ and take the corresponding $\alpha\in (0,1)$ so that \eqref{goal:Ainfty} holds.
	We are going to see that
	\begin{equation}\label{Ainfty-daydic}
	\frac{\omega_L^{X_{Q^0}}(F)}{\omega_L^{X_{Q^0}}(Q_0)}\le \frac{\alpha}{C_1}
	\implies
	\frac{\sigma(F)}{\sigma(Q_0)}\le\beta.
	\end{equation}
	Assuming that the first estimate holds  we see that \eqref{cop-dyadic} yields
	$\frac{\omega_L^{X_{Q_0}}(F)}{\omega_L^{X_{Q_0}}(Q_0)}\le \alpha$. Thus we can apply \eqref{goal:Ainfty} to obtain that $\frac{\sigma(F)}{\sigma(Q_0)}\le\beta$ as desired.  
	
To complete the proof we need to see that \eqref{Ainfty-daydic} gives \eqref{Ainfty-alpha-beta}. We show its contrapositive. Fix $\beta\in (0,1)$ and  a surface ball $\Delta_0=B_0\cap\partial\Omega$, with $B_0=B(x_0,r_0)$, $x_0\in\partial\Omega$, and $0<r_0<\diam(\partial\Omega)$. Take an arbitrary surface ball $\Delta=B\cap\partial\Omega$ centered at $\partial\Omega$ with $B=B(x,r)\subset B_0$, and let $F\subset\Delta$ be a Borel set such that 
${\sigma(F)}>\beta{\sigma(\Delta)}$. Consider the pairwise disjoint family $\mathcal{F}=\{Q\in \mathbb{D}: Q\cap \Delta\neq\emptyset, \frac{r}{4\,C}<\ell(Q)\le \frac{r}{2\,C}\}$ where $C$ is the constant in \eqref{deltaQ}. In particular,  $\Delta\subset \cup_{\mathcal{F}} Q\subset 2\Delta$. The pigeon-hole principle yields that there is a constant $C'>1$ depending just on the Ahlfors regularity constant of $\sigma$ so that $\frac{\sigma(F\cap Q_0)}{\sigma(Q_0)}>\frac{\beta}{C'}$ for some $Q_0\in\mathcal{F}$. Let $Q^0\in\mathbb{D}$ be the unique dyadic cube such that $Q_0\subset Q^0$ and  $\frac{r_0}{2}<\ell(Q^0)\le r_0$. We can then invoke \eqref{Ainfty-daydic} with $\frac{\beta}{C'}$   to find $\alpha\in (0,1)$ such that by Lemma \ref{proppde}, and Harnack's inequality
\[
\frac{\omega_L^{X_{\Delta_0}}(F)}{\omega_L^{X_{\Delta_0}}(\Delta)}
\ge
\frac{\omega_L^{X_{\Delta_0}}(F\cap Q_0)}{\omega_L^{X_{\Delta_0}}(\Delta)}
\approx
\frac{\omega_L^{X_{\Delta_0}}(F\cap Q_0)}{\omega_L^{X_{\Delta_0}}(Q_0)}
\approx
\frac{\omega_L^{X_{Q^0}}(F\cap Q_0)}{\omega_L^{X_{Q^0}}(Q_0)}
>\frac{\alpha}{C_1}.
\]
In short, we have obtained that for every $\beta\in (0,1)$ there exists $\widetilde{\alpha}\in (0,1)$ such that
\[
\frac{\sigma(F)}{\sigma(\Delta)}>\beta
\implies
\frac{\omega_L^{X_{\Delta_0}}(F)}{\omega_L^{X_{\Delta_0}}(\Delta)}>\widetilde{\alpha},
\]
which is the contrapositive of \eqref{Ainfty-alpha-beta}. This completes the proof of Theorem \ref{theor:cme-implies-ainf} modulo the proof of Lemma \ref{sq-function->M}.
 \end{proof}

Before proving Lemma \ref{sq-function->M} we need some notation and some estimates. Let $\eta=2^{-k_*}<1$. 
\begin{eqnarray}\label{tt-1}
&\hbox{Given }Q\in\mathbb{D}(\partial\Omega)\hbox{ we define }\widetilde{Q}\in\mathbb{D}_Q\hbox{ to be the unique cube }\nonumber\\ 
&\qquad\hbox{such that }x_Q\in\widetilde{Q},\hbox{ and }\ell(\widetilde{Q})=\eta\ell(Q).
\end{eqnarray}
Using this notation we have the following estimates which will be used later:
\begin{equation}\label{est-HC-daydic}
\omega_L^{X_{\widetilde{Q}}}(\partial\Omega\setminus Q)
=
\omega_L^{X_{\widetilde{Q}}}(\partial\Omega)-\omega_L^{X_{\widetilde{Q}}}(Q)
\le
1-\omega_L^{X_{\widetilde{Q}}}(Q)
\le C\eta^\gamma
\end{equation}
where $C$ depends on dimension, the $1$-sided $\mathrm{CAD}$ constants and the ellipticity of $L$ and $\gamma$ is the parameter in Lemma \ref{proppde}. To see this, keeping in mind the notation introduced in \eqref{deltaQ}, let $\varphi(X)=\varphi_0((X-x_Q)/r_Q)$ where $\varphi_0\in C_c(\re^{n+1})$ with $\mathbf{1}_{B(0,1)}\le\varphi_0\le \mathbf{1}_{B(0,2)}$. Note that $\varphi\in C_c(\re^{n+1})$ with $0\le \varphi\le 1$, $\supp(\varphi)\subset 2B_Q$,  and $\varphi\equiv 1$ in $B_Q$. In particular, $\varphi\rest{\partial\Omega}\le \mathbf{1}_{2\Delta_Q}\le \mathbf{1}_{Q}$
and hence
\begin{equation}\label{qweaffer}
v(X):=
\int_{\partial\Omega} \varphi(y)d\omega_L^{X_{\widetilde{Q}}}(y)
\le
\omega_L^{X_{\widetilde{Q}}}(Q)
\end{equation}
Note that $v\in W^{1,2}_{\rm loc}(\Omega)\cap C(\overline{\Omega})$ is a weak solution with $0\le v\le 1$ and $v\rest{\partial\Omega}=\varphi\rest{\partial\Omega}\equiv 1$  in $B_Q$. Thus, $\widetilde{v}=1-v\in W^{1,2}_{\rm loc}(\Omega)\cap C(\overline{\Omega})$ is a weak solution with $0\le \widetilde{v} \le 1$ and $\widetilde{v}\rest{\partial\Omega}=1-\varphi\rest{\partial\Omega}\equiv 0$  in $B_Q$. Thus we can use \eqref{qweaffer} and  part $(a)$ in Lemma \ref{proppde} to see that
\begin{equation}
1-\omega_L^{X_{\widetilde{Q}}}(Q)
\le
1-v(X)
=
\widetilde{v}(X)
\lesssim
\left(\frac{|X_{\widetilde{Q}}-x_Q|}{r_Q}\right)^\gamma \|\widetilde{v}\|_{L^\infty(\Omega)}
\le
C\eta^\gamma,
\end{equation}
where the last estimate follows from
\[
|X_{\widetilde{Q}}-x_Q|
\le
|X_{\widetilde{Q}}-x_{\widetilde{Q}}|+|x_{\widetilde{Q}}-x_Q|
\lesssim
\ell(\widetilde{Q})
=\eta \ell(Q),
\]
since $x_Q\in \widetilde{Q}$ and $X_{\widetilde{Q}}$ is a corkscrew point relative to $\widetilde{Q}$.

We also claim that there exists $c_0\in (0,1)$ depending only on the $\mathrm{AR}$ constant and on the ellipticity of $L$ so that if $\eta$ is small enough (depending only on $n$ and the AR constant) then
\begin{equation}\label{c0c1}
c_0\le \omega_L^{X_{\widetilde{Q}}}(\widetilde{Q})\le 1-c_0.
\end{equation}
The first inequality follows at once from Lemma \ref{bourgain} and Harnack's inequality. For the second one we claim that if $\eta$ is small enough we can find $\widetilde{Q}'\in\mathbb{D}$ with $\ell(\widetilde{Q}')=\ell(\widetilde{Q})$, $\widetilde{Q}'\cap\widetilde{Q}=\emptyset$ and $\dist(\widetilde{Q},\widetilde{Q}')\lesssim\ell(\widetilde{Q})$. Indeed, if we write $\widetilde{Q}^j$ for the $j$-th ancestor of $\widetilde{Q}$ (that is, the unique cube satisfying $\ell( \widetilde{Q}^j)=2^j \ell(\widetilde{Q})$ and $\widetilde{Q}\subset \widetilde{Q}^j$) then
$\sigma(\widetilde{Q}^j)\gtrsim \ell(\widetilde{Q}^j)^n= 2^{jn}\ell(\widetilde{Q})^n >\sigma(\widetilde{Q})$ for $j$ large enough depending on the  $\mathrm{AR}$ constant. Note that in the previous estimates we are implicitly using that $\ell(\widetilde{Q})\lesssim\diam(\partial\Omega)$, fact that follows by choosing $\eta$ small enough depending on the  $\mathrm{AR}$ constant. Once $j$ has been chosen we must have $\widetilde{Q}\subsetneq\widetilde{Q}^j$, and we can easily pick $\widetilde{Q}'\in\mathbb{D}_{\widetilde{Q}^j}$ with all the desired properties. In turn by Harnack's inequality and Lemma \ref{bourgain} one can see that $\omega^{X_{\widetilde{Q}}}(\widetilde{Q}')\gtrsim\omega^{X_{\widetilde{Q}'}}(\widetilde{Q}')\ge C^{-1}$ with $C>1$ and consequently
\[
\omega_L^{X_{\widetilde{Q}}}(\widetilde{Q})
=
\omega_L^{X_{\widetilde{Q}}}(\partial\Omega)-\omega_L^{X_{\widetilde{Q}}}(\partial\Omega\setminus\widetilde{Q})
\le
1-\omega_L^{X_{\widetilde{Q}}}(\widetilde{Q}')
\le
1-C^{-1},
\]
which is the desired estimate.

\begin{proof}[Proof of Lemma \ref{sq-function->M}]
	Let $\eta=2^{-k_*}<1$ be a small dyadic number to be chosen and such that \eqref{est-HC-daydic} and \eqref{c0c1} hold. Fix $Q_0\in\mathbb{D}$ and note that $\omega:=\omega_L^{X_{Q_0}}$ is a regular Borel measure on $\partial \Omega$ which is dyadically doubling with constants $C_0$ (depending only on dimension, the $1$-sided $\mathrm{CAD}$ constants and the ellipticity of $L$) by part $(c)$ of Lemma \ref{proppde} and Harnack's inequality.
	Let $0<\varepsilon_0<e^{-1}$ and $0<\alpha< \varepsilon_0^2/(2C_0^2)$, sufficiently small to be chosen later, and let $F\subset Q_0$ be a Borel set such that $\omega(F)\leq\alpha\omega(Q_0)$. By
	Lemma \ref{extract-eps-cover} applied to $\mu=\omega$, it follows that $F$ has a good $\varepsilon_0$-cover of length $k\approx\tfrac{\log{\alpha^{-1}}}{\log{\varepsilon_0^{-1}}}$, with $k\geq 2$. Let $\{\mathcal{O}_\ell\}_{\ell=1}^{k}$ be the corresponding collection of Borel sets so that $F\subset \mathcal{O}_k\subset\cdots\subset\mathcal{O}_1\subset Q_0$ and $\mathcal{O}_\ell=\bigcup_{Q_i^\ell\in\mathcal{F}_\ell}Q_i^{\ell}$, with disjoint families $\mathcal{F}_\ell=\{Q_i^\ell\}\subset\mathbb{D}_{Q_0}\setminus\{Q_0\}$. Now, using the notation above (see \eqref{tt-1}) we define $\widetilde{\mathcal{O}}_\ell:=\bigcup_{Q_i^\ell\in\mathcal{F}_\ell}\widetilde{Q}_i^\ell$ and consider the Borel set $S:=\bigcup_{j=2}^k\big(\widetilde{\mathcal{O}}_{j-1}\setminus\mathcal{O}_j\big)$.
	Note that the union of sets comprising $S$ is disjoint, hence
	\begin{equation}\label{def-G}
	\mathbf{1}_S(y)=\sum_{j=2}^k\mathbf{1}_{\widetilde{\mathcal{O}}_{j-1}\setminus\mathcal{O}_j}(y),\qquad y\in\partial\Omega.
	\end{equation}

	Now we introduce some notation. For each $y\in F$ and $1\leq\ell\leq k$, there exists a unique $Q_i^\ell(y)\in\mathcal{F}_\ell$ such that $y\in Q_i^{\ell}(y)$. Let $P_i^\ell(y)\in\mathbb{D}_{Q_i^\ell(y)}$ be the unique cube verifying  $y\in P_i^\ell(y)$ and $\ell(P_i^\ell(y))=\eta\ell(Q_i^\ell(y))$. Associated with $P_i^\ell(y)$ we can construct $\widetilde{P}_i^\ell(y)$ as above (see \eqref{tt-1}), that is, $\widetilde{P}_i^\ell(y)\in\mathbb{D}_{P_i^\ell(y)}$ satisfies $\ell(\widetilde{P}_i^{\ell}(y))=\eta\ell(P_i^\ell(y))$ and $x_{P_i^{\ell}(y)}\in \widetilde{P}_i^{\ell}(y)$, where $x_{P_i^\ell(y)}$ is the center of $P_i^\ell(y)$. As usual we write $X_{\widetilde{Q}_i^\ell(y)}$ and $X_{\widetilde{P}_i^\ell(y)}$ to denote, respectively, the corkscrew points associated to $\widetilde{Q}_i^\ell(y)$ and $\widetilde{P}_i^\ell(y)$.

	Let $u(X):=\omega^X_L(S)$ then
	\begin{equation}\label{def-u}
	u(X)=\int_{\partial\Omega}\mathbf{1}_S(y)\,d\omega^X_L(y)=\sum_{j=2}^k\omega^X_L(\widetilde{\mathcal{O}}_{j-1}\setminus\mathcal{O}_j).
	\end{equation}
	The following lemma contains a lower bound for the oscillation of $u$. Here $\eta$ is as in \eqref{tt-1} and $F$ which was used to construct $S$ (as above) has a good $\varepsilon_0$-cover.

	\begin{lemma}\label{claim2}
		If $\eta$ and $\varepsilon_0$ are taken sufficiently small (depending only on $n$, the $1$-sided $\mathrm{CAD}$ constants and the ellipticity of $L$), then for each $y\in F$, and each $1\leq\ell\leq k-1$, we have that
		\begin{equation}\label{oscillation}
		\big|u(X_{\widetilde{Q}_i^\ell(y)})-u(X_{\widetilde{P}_i^\ell(y)})\big|\geq \frac{c_0}2,
		\end{equation}
		where $c_0$ is the constant in  \eqref{c0c1}
	\end{lemma}
	
	Assume this result for now and continue the proof of  Lemma \ref{sq-function->M}. Fix $\eta$ and $\varepsilon_0$ as in Lemma \ref{claim2}. Fix also $y\in F$, $1\leq\ell\leq k-1$, and write $Q_i^\ell:=Q_i^\ell(y)\in\mathbb{D}_{Q_0}$, and $P_i^\ell:=P_i^\ell(y)\in\mathbb{D}_{Q_i^\ell}$ using the notation above. By construction  $X_{\widetilde{Q}_i^\ell}\in U_{\widetilde{Q}_i^\ell}$
	and $X_{\widetilde{P}_i^\ell}\in U_{\widetilde{P}_i^\ell}$, hence we can find Whitney cubes $I_{\widetilde{Q}_i^\ell}\in\mathcal{W}_{\widetilde{Q}_i^{\ell}}^*$ and $I_{\widetilde{P}_i^\ell}\in\mathcal{W}_{\widetilde{P}_i^{\ell}}^*$ so that $X_{\widetilde{Q}_i^\ell}\in I_{\widetilde{Q}_i^\ell}$ and $X_{\widetilde{P}_i^\ell}\in I_{\widetilde{P}_i^\ell}$.
	
	Also, note that $\ell(\widetilde{Q}_i^\ell)=\eta\ell(Q_i^\ell)$ and $\ell(\widetilde{P}_i^\ell)=\eta^2\ell(Q_i^\ell)$ which imply
	$\ell(\widetilde{Q}_i^\ell)>\ell(\widetilde{P}_i^\ell)>\eta^3\ell(Q_i^\ell)$ since $\eta<1$. On the other hand, $\widetilde{Q}_i^\ell \subset Q_i^\ell$ and  $\widetilde{P}_i^\ell \subset P_i^\ell\subset Q_i^\ell$, which in turn yield that $I_{\widetilde{Q}_i^\ell}^*$ and $I_{\widetilde{P}_i^\ell}^*$ are both contained in $U_{Q_i^\ell,\eta^3}$. Using \eqref{oscillation}, the notation $[u]_{U_{Q_i^\ell,\eta^3}}:=\bariint_{U_{Q_i^\ell,\eta^3}}udX$, Moser's ``local boundedness''  estimates and the previous observations we can obtain
	\begin{align*}
	\frac{c_0}2&\leq\big|u(X_{\widetilde{Q}_i^\ell})-[u]_{U_{Q_i^\ell,\eta^3}}\big|+\big|[u]_{U_{Q_i^\ell,\eta^3}}-u(X_{\widetilde{P}_i^\ell})\big|
	\\
	&\lesssim\bigg(\bariint_{I_{\widetilde{Q}_i^\ell}^*}\big|u(Y)-[u]_{U_{Q_i^\ell,\eta^3}}\big|^2\,dY \bigg)^{1/2} + \bigg(\bariint_{I_{\widetilde{P}_i^\ell}^*}\big|u(Y)-[u]_{U_{Q_i^\ell,\eta^3}}\big|^2\,dY \bigg)^{1/2}
	\\
	&\leq C_\eta\bigg(\ell(Q_i^\ell)^{-n-1}\iint_{U_{Q_i^\ell,\eta^3}}\big|u(Y)-[u]_{U_{Q_i^\ell,\eta^3}} \big|^2\,dY\bigg)^{1/2}
	\\
	&\leq C_\eta\bigg(\iint_{U_{Q_i^\ell,\eta^3}}|\nabla u(Y)|^2\delta(Y)^{1-n}\,dY\bigg)^{1/2},
	\end{align*}
	where the last estimate follows from the Poincaré's inequality in \cite[Lemma 3.1]{HMT-var}, and the fact that $\delta(Y)\approx_\eta \ell(Q_i^\ell)$ for every $Y\in U_{Q_i^\ell,\eta^3}$. Summing up the above estimate, taking into account that  the sets $\{U_{Q,\eta^3}\}_{Q\in\mathbb{D}_{Q_0}}$ have bounded overlap with constant depending on $\eta$, and using Lemma \ref{extract-eps-cover}, we obtain if $\alpha$ is small enough
	$$
	\frac{c_0^2}{4}\frac{\log{\alpha^{-1}}}{\log{\varepsilon_0^{-1}}}
	\approx
	\frac{c_0^2}{4}\,(k-1)
	\le  C_\eta\sum_{\ell=1}^{k-1}\iint_{U_{Q_i^\ell,\eta^3}}|\nabla u(Y)|^2\delta(Y)^{1-n}\,dY\le C_\eta \big( S_{Q_0}^\eta(u)(y)\big)^2.
	$$
	This completes the proof of Lemma \ref{sq-function->M}.
\end{proof}

\begin{proof}[Proof of Lemma \ref{claim2}]
	Fix $y\in F$ and write $Q_i^\ell:=Q_i^\ell(y)$, $P_i^\ell:=P_i^\ell(y)$. Our first goal is to estimate $u(X_{\widetilde{Q}_i^\ell})$. By  \eqref{est-HC-daydic} and using \eqref{def-u} we have
	\begin{multline}\label{decomp-u(XQ)}
	u(X_{\widetilde{Q}_i^\ell})=\omega_L^{X_{\widetilde{Q}_i^\ell}}(S)	
	\le \omega_L^{X_{\widetilde{Q}_i^\ell}}(\partial\Omega\setminus Q_i^\ell)+\omega_L^{X_{\widetilde{Q}_i^\ell}}(S\cap Q_i^\ell)
	\\
	\le C\eta ^\gamma+\omega_L^{X_{\widetilde{Q}_i^\ell}}(S\cap Q_i^\ell)	
	=:C\eta ^\gamma +\mathrm{I}.
	\end{multline}
	For $1\le \ell\le k-1$ we have that  $Q_i^\ell\subset\mathcal{O}_\ell\subset\mathcal{O}_j$ for each $2\leq j\leq\ell$ and hence by \eqref{def-G} we have 
	\begin{multline}\label{I:I1-I2}
	\mathrm{I}=
	\sum_{j=2}^k\omega_{L}^{X_{\widetilde{Q}_i^\ell}}\big(Q_i^\ell\cap(\widetilde{\mathcal{O}}_{j-1}\setminus\mathcal{O}_j)\big)
	=
	\sum_{j=\ell+1}^k\omega_{L}^{X_{\widetilde{Q}_i^\ell}}\big(Q_i^\ell\cap(\widetilde{\mathcal{O}}_{j-1}\setminus\mathcal{O}_j)\big)
	\\
	=
	\sum_{j=\ell+2}^k\omega_{L}^{X_{\widetilde{Q}_i^\ell}}\big(Q_i^\ell\cap(\widetilde{\mathcal{O}}_{j-1}\setminus\mathcal{O}_j)\big)
	+\omega_{L}^{X_{\widetilde{Q}_i^\ell}}\big(Q_i^\ell\cap(\widetilde{\mathcal{O}}_{\ell}\setminus\mathcal{O}_{\ell+1})\big)=:\mathrm{I}_1+\mathrm{I}_2,
	\end{multline}
	with the understanding that if $\ell=k-1$ then $\mathrm{I}_1=0$.
	
	Next, we claim that $\mathrm{I}_1\leq C_\eta\varepsilon_0$. This is clear if $\ell=k-1$. For $1\le \ell\le k-2$, using Harnack's inequality to move from $X_{\widetilde{Q}_i^\ell}$ to $X_{Q_i^\ell}$ (with constants depending on $\eta$), Lemma \ref{proppde} parts $(c)$ and $(d)$ (recall that $\omega=\omega_L^{X_{Q_0}}$), we have that
	\begin{multline}\label{I1}
	\mathrm{I}_1\leq C_\eta\sum_{j=\ell+2}^k\omega_{L}^{X_{Q_i^\ell}}\big(Q_i^\ell\cap(\widetilde{\mathcal{O}}_{j-1}\setminus\mathcal{O}_j)\big)
	\leq\frac{C_\eta}{\omega(Q_i^\ell)}\sum_{j=\ell+2}^k\omega\big(Q_i^\ell\cap(\widetilde{\mathcal{O}}_{j-1}\setminus\mathcal{O}_j)\big)
	\\
	\leq\frac{C_\eta}{\omega(Q_i^\ell)}\sum_{j=\ell+2}^k\omega(Q_i^\ell\cap\mathcal{O}_{j-1})\leq
	C_\eta\sum_{j=\ell+2}^k\varepsilon_0^{j-1-\ell}\leq C_\eta\varepsilon_0,
	\end{multline}
	where the next-to-last estimate follows from Lemma \ref{lemma-eps-cover} with $\mu=\omega$, and the last one uses that $\varepsilon_0<e^{-1}$.
	Let us now focus on $\mathrm{I}_2$.  Note that $Q_i^\ell\cap\widetilde{\mathcal{O}}_\ell=\widetilde{Q}_i^\ell$, hence \eqref{c0c1} yields
	\[
	\mathrm{I}_2=\omega_L^{X_{\widetilde{Q}_i^\ell}}(\widetilde{Q}_i^\ell\setminus\mathcal{O}_{\ell+1})
	\le\omega_L^{X_{\widetilde{Q}_i^\ell}}(\widetilde{Q}_i^\ell)\le 1-c_0
	.
	\]
	Collecting this with \eqref{decomp-u(XQ)}, \eqref{I:I1-I2}, \eqref{I1}, we conclude that
	\begin{equation}\label{u-Q-above}
	u(X_{\widetilde{Q}_i^\ell})
	\le
	C\eta^\gamma +C_\eta\varepsilon_0+1-c_0
	\le
	1-\frac34c_0,
	\end{equation}	
	by choosing first $\eta$ small enough so that $C\eta^\gamma<c_0/8$ and then $\varepsilon_0$ small enough so that $C_\eta\varepsilon_0<c_0/8$.

	To get a lower bound for $u(X_{\widetilde{Q}_i^\ell})$ we use that
	$Q_i^\ell\cap\widetilde{\mathcal{O}}_\ell=\widetilde{Q}_i^\ell$ and \eqref{c0c1}:
	\begin{multline*}
	u(X_{\widetilde{Q}_i^\ell})=\omega_L^{X_{\widetilde{Q}_i^\ell}}(S)
	\ge
	\omega_{L}^{X_{\widetilde{Q}_i^\ell}}\big(Q_i^\ell\cap(\widetilde{\mathcal{O}}_{\ell}\setminus\mathcal{O}_{\ell+1})\big)
	\\
	=
	\omega_L^{X_{\widetilde{Q}_i^\ell}}(\widetilde{Q}_i^\ell\setminus\mathcal{O}_{\ell+1})
	=
	\omega_L^{X_{\widetilde{Q}_i^\ell}}(\widetilde{Q}_i^\ell)-\omega_L^{X_{\widetilde{Q}_i^\ell}}(\widetilde{Q}_i^\ell\cap\mathcal{O}_{\ell+1})
	\ge c_0-\omega_L^{X_{\widetilde{Q}_i^\ell}}(\widetilde{Q}_i^\ell\cap\mathcal{O}_{\ell+1}).
	\end{multline*}
	Using Harnack's inequality to move from $X_{\widetilde{Q}_i^\ell}$ to $X_{Q_i^\ell}$ (with constants depending on $\eta$), Lemma \ref{proppde} parts $(c)$ and $(d)$ (recall that $\omega=\omega_L^{X_{Q_0}}$), we have that
	\begin{equation}\label{tgwegte}
	\omega_L^{X_{\widetilde{Q}_i^\ell}}(\widetilde{Q}_i^\ell\cap\mathcal{O}_{\ell+1})
	\leq C_\eta\omega_L^{X_{Q_i^\ell}}(Q_i^\ell\cap\mathcal{O}_{\ell+1})
	\leq{C_\eta} \frac{\omega(Q_i^\ell\cap\mathcal{O}_{\ell+1})}{\omega(Q_i^\ell)}
	\leq
	C_\eta\varepsilon_0,
	\end{equation}
	where the last estimate follows from Lemma \ref{lemma-eps-cover} with $\mu=\omega$ and since $1\le \ell \le k-1$. Assuming further that $C_\eta\varepsilon_0<c_0/4$ we arrive at
	\begin{equation}\label{u-Q-below}
	u(X_{\widetilde{Q}_i^\ell})
	\ge
	c_0-C_\eta\varepsilon_0\ge\frac34 c_0.
	\end{equation}
	
	Let us now focus on estimating $u(X_{\widetilde{P}_i^\ell})$ and we consider two cases:

	\noindent\textbf{Case 1:} $P_i^\ell\cap\widetilde{Q}_i^\ell=\emptyset$.  Much as before by  \eqref{est-HC-daydic}
	\begin{multline}\label{decomp-u(XP)}
	u(X_{\widetilde{P}_i^\ell})=\omega_L^{X_{\widetilde{P}_i^\ell}}(S)	
	\le \omega_L^{X_{\widetilde{P}_i^\ell}}(\partial\Omega\setminus P_i^\ell)+\omega_L^{X_{\widetilde{P}_i^\ell}}(S\cap P_i^\ell)
	\\
	\le C\eta ^\gamma+\omega_L^{X_{\widetilde{P}_i^\ell}}(S\cap P_i^\ell)	
	=:C\eta ^\gamma +\widehat{\mathrm{I}}.
	\end{multline}
	For $1\le \ell\le k-1$ we have that  $P_i^\ell\subset Q_i^\ell \subset\mathcal{O}_\ell\subset\mathcal{O}_j$ for each $2\leq j\leq\ell$ and hence
	\begin{multline}\label{I:I1-I2:hat}
	\widehat{\mathrm{I}}=
	\sum_{j=2}^k\omega_{L}^{X_{\widetilde{P}_i^\ell}}\big(P_i^\ell\cap(\widetilde{\mathcal{O}}_{j-1}\setminus\mathcal{O}_j)\big)
	=
	\sum_{j=\ell+1}^k\omega_{L}^{X_{\widetilde{P}_i^\ell}}\big(P_i^\ell\cap(\widetilde{\mathcal{O}}_{j-1}\setminus\mathcal{O}_j)\big)
	\\
	=
	\sum_{j=\ell+2}^k\omega_{L}^{X_{\widetilde{P}_i^\ell}}\big(P_i^\ell\cap(\widetilde{\mathcal{O}}_{j-1}\setminus\mathcal{O}_j)\big)
	+\omega_{L}^{X_{\widetilde{P}_i^\ell}}\big(P_i^\ell\cap(\widetilde{\mathcal{O}}_{\ell}\setminus\mathcal{O}_{\ell+1})\big)
	=:
	\widehat{\mathrm{I}}_1+\widehat{\mathrm{I}}_2,
	\end{multline}
	with the understanding that if $\ell=k-1$ then $\widehat{\mathrm{I}}_1=0$. The estimate for $\widehat{\mathrm{I}}_1$ (when $\ell\le k-2$) follows from that of $\mathrm{I}_1$ since
	using Harnack's inequality to move from $X_{\widetilde{P}_i^\ell}$ to $X_{\widetilde{Q}_i^\ell}$ and the fact that $P_i^\ell\subset Q_i^\ell$ we easily obtain from \eqref{I1}
	\begin{equation}\label{I1:hat}
	\widehat{\mathrm{I}}_1
	\le
	C_\eta
	\sum_{j=\ell+2}^k\omega_{L}^{X_{\widetilde{Q}_i^\ell}}\big(Q_i^\ell\cap(\widetilde{\mathcal{O}}_{j-1}\setminus\mathcal{O}_j)\big)
	=
	C_\eta \mathrm{I}_1
	\le
	C_\eta  \varepsilon_0.
	\end{equation}
	On the other hand, note that $P_i^\ell\cap(\widetilde{\mathcal{O}}_\ell\setminus\mathcal{O}_{\ell+1})=(P_i^\ell\cap \widetilde{Q}_i^\ell)\setminus\mathcal{O}_{\ell+1}\subset P_i^\ell\cap \widetilde{Q}_i^\ell=\emptyset$ and hence $\widehat{\mathrm{I}}_2=0$. Thus \eqref{decomp-u(XP)}, \eqref{I:I1-I2:hat}, and \eqref{I1:hat} yield
	\begin{equation}\label{conclusion1}
	u(X_{\widetilde{P}_i^\ell})\leq C\eta^\gamma+C_\eta\varepsilon_0\leq\frac{1}{4}c_0,
	\end{equation}
	by choosing first $\eta$ small enough so that $C\eta^\gamma<c_0/8$ and then $\varepsilon_0$ small enough so that $C_\eta\varepsilon_0<c_0/8$. This estimate along with \eqref{u-Q-below} give at once
	\[
	|u(X_{\widetilde{Q}_i^\ell})-u(X_{\widetilde{P}_i^\ell})|=u(X_{\widetilde{Q}_i^\ell})-u(X_{\widetilde{P}_i^\ell})\geq \frac{3}{4}c_0-\frac{1}{4}c_0=\frac{1}{2}c_0,
	\]
	which is the desired estimate.
	
	\medskip

	\noindent\textbf{Case 2:} $P_i^\ell\cap\widetilde{Q}_i^\ell\neq\emptyset$.  Notice that since both cubes have the same sidelength it follows that $P_i^\ell=\widetilde{Q}_i^\ell$. Our goal is to get a lower bound for $u(X_{\widetilde{P}_i^\ell})$. We use that
	$P_i^\ell\cap\widetilde{\mathcal{O}}_\ell= \widetilde{Q}_i^\ell\cap\widetilde{\mathcal{O}}_\ell= \widetilde{Q}_i^\ell=P_i^\ell$ and \eqref{est-HC-daydic}:
	\begin{multline*}
	u(X_{\widetilde{P}_i^\ell})=\omega_L^{X_{\widetilde{P}_i^\ell}}(S)
	\ge
	\omega_{L}^{X_{\widetilde{P}_i^\ell}}\big(P_i^\ell\cap(\widetilde{\mathcal{O}}_{\ell}\setminus\mathcal{O}_{\ell+1})\big)
	=
	\omega_L^{X_{\widetilde{P}_i^\ell}}(P_i^\ell\setminus\mathcal{O}_{\ell+1})
	\\
	=
	\omega_L^{X_{\widetilde{P}_i^\ell}}(P_i^\ell)-\omega_L^{X_{\widetilde{P}_i^\ell}}(P_i^\ell\cap\mathcal{O}_{\ell+1})
	\ge
	1-C\eta^\gamma-\omega_L^{X_{\widetilde{P}_i^\ell}}(P_i^\ell\cap\mathcal{O}_{\ell+1}).
	\end{multline*}
	Moreover, using Harnack's inequality to move from $X_{\widetilde{P}_i^\ell}$ to $X_{\widetilde{Q}_i^\ell}$ (with constants depending on $\eta$) and \eqref{tgwegte} we observe that
	\[
	\omega_L^{X_{\widetilde{P}_i^\ell}}(P_i^\ell\cap\mathcal{O}_{\ell+1})
	=
	\omega_L^{X_{\widetilde{P}_i^\ell}}(\widetilde{Q}_i^\ell\cap\mathcal{O}_{\ell+1})
	\le
	C_\eta
	\omega_L^{X_{\widetilde{Q}_i^\ell}}(\widetilde{Q}_i^\ell\cap\mathcal{O}_{\ell+1})
	\le
	C_\eta\varepsilon_0.
	\]
	Collecting the obtained estimates we conclude that
	\begin{equation}\label{u-P-below}
	u(X_{\widetilde{Q}_i^\ell})
	\ge
	1-C\eta^\gamma-C_\eta\varepsilon_0
	\ge
	1-\frac14 c_0,
	\end{equation}
	if we choose first $\eta$ small enough so that $C\eta^\gamma<c_0/8$ and then $\varepsilon_0$ small enough so that $C_\eta\varepsilon_0<c_0/8$.
	If we now gather \eqref{u-Q-above} and \eqref{u-P-below} we eventually obtain the desired estimate
	\[
	\big|u(X_{\widetilde{Q}_i^\ell})-u(X_{\widetilde{P}_i^\ell})\big|=u(X_{\widetilde{P}_i^\ell})-u(X_{\widetilde{Q}_i^\ell})\geq \bigg(1-\frac{1}{4}c_0\bigg)-\bigg(1-\frac{3}{4}c_0\bigg)=\frac12 c_0.
	\]
	This completes the proof.
\end{proof}

\subsection{$A_\infty$ implies the Carleson measure condition}\label{section:Proof-CME:b->a}

The proof of Theorem \ref{theor:cme-implies-ainf}: $(b)\Longrightarrow (a)$ requires 
 some additional notation and several auxiliary results. 
 
 Let $Q_0\in\mathbb{D}$ and $\alpha=\{\alpha_Q\}_{Q\in\mathbb{D}_{Q_0}}$ be a sequence of non-negative numbers indexed by the dyadic cubes in $\mathbb{D}_{Q_0}$. For any collection $\mathbb{D}'\subset\mathbb{D}_{Q_0}$, we define the associated discrete ``measure''
\begin{equation}\label{defmeasure}
\mathfrak{m}_\alpha(\mathbb{D}'):=\sum_{Q\in\mathbb{D}'}\alpha_Q.
\end{equation}
We say that $\mathfrak{m}_\alpha$ is a discrete ``Carleson measure'' (with respect to $\sigma$) in $Q_0$, if
\begin{equation}\label{defcarlesonmeasure}
\|\mathfrak{m}_\alpha\|_{\mathcal{C}(Q_0)}:=\sup_{Q\in\mathbb{D}_{Q_0}}\frac{\mathfrak{m}_\alpha(\mathbb{D}_Q)}{\sigma(Q)}<\infty.
\end{equation}

The following result reduces the desired Carleson measure estimate to a discrete one:

\begin{lemma}\label{lemmareduction}
	Let $\Omega\subset\re^{n+1}$ be a 1-sided $\mathrm{CAD}$ and let $Lu=-\div(A\nabla u)$ be a real (not necessarily symmetric) elliptic operator. Let $u\in W^{1,2}_{\rm loc}(\Omega)\cap L^\infty(\Omega)$ satisfy $Lu=0$ in the weak sense in $\Omega$ and define
	\begin{equation}\label{defcoefficients}
	\alpha:=\{\alpha_Q\}_{Q\in\mathbb{D}}
	:=\Big\{\iint_{U_Q}|\nabla u(X)|^2\delta(X)\,dX\Big\}_{Q\in\mathbb{D}}.
	\end{equation}
	Suppose that there exist $C_0,M_0\geq 1$ such that $\|\mathfrak{m}_\alpha\|_{\mathcal{C}(Q)}\leq C_0\|u\|_{L^{\infty}(\Omega)}^2$ for every $Q\in\mathbb{D}(\partial\Omega)$ verifying $\ell(Q)<\diam(\partial\Omega)/M_0$. Then,
	\begin{equation}\label{cmeestimate:new}
	\sup_{\substack{x\in\partial\Omega \\ 0<r<\infty}}\frac{1}{r^n}\iint_{B(x,r)\cap\Omega}|\nabla u(X)|^2\delta(X)\,dX\leq C(1+C_0+M_0)\|u\|_{L^\infty(\Omega)}^2,
	\end{equation}
	where $C$ depends only on dimension, the 1-sided $\mathrm{CAD}$ constants, and the ellipticity of $L$.
\end{lemma}
\begin{proof}
	By homogeneity we may assume that $\|u\|_{L^{\infty}(\Omega)}=1$.
	First, we claim that
	\begin{equation}\label{dyadiccmeestimate}
	\sup_{Q\in\mathbb{D}(\partial\Omega)}\frac{1}{\sigma(Q)}\iint_{T_{Q}}|\nabla u(X)|^2\delta(X)\,dX\lesssim C_0+M_0.
	\end{equation}
	Given $Q_0\in\mathbb{D}(\partial\Omega)$ such that $\ell(Q_0)<\diam(\partial\Omega)/M_0$, we have that
	$$
	\iint_{T_{Q_0}}|\nabla u(X)|^2\delta(X)\,dX
	\leq\sum_{Q\in\mathbb{D}_{Q_0}}\alpha_Q= \mathfrak{m}_\alpha(\mathbb{D}_{Q_0})
	\leq
	\|\mathfrak{m}_\alpha\|_{\mathcal{C}(Q_0)}\sigma(Q_0)
	\leq
	C_0\sigma(Q_0).
	$$
	Otherwise, if $\ell(Q_0)\geq\diam(\partial\Omega)/M_0$ (this happens only if $\diam(\partial\Omega)<\infty)$, there exists a unique $k_0\ge 1$ so that
	\[
	2^{k_0-1}\frac{\diam(\partial\Omega)}{M_0}
	\le
	\ell(Q_0)
	<2^{k_0}\frac{\diam(\partial\Omega)}{M_0}.
	\]
	As observed before if $\diam(\partial\Omega)<\infty$ then $\ell(Q_0)\lesssim \diam(\partial\Omega)$ hence $2^{k_0}\lesssim M_0$. Define the disjoint collection $
	\mathcal{D}_{0}:=\big\{Q'\in\mathbb{D}_{Q_0}:\:\ell(Q')=2^{-k_0}\ell(Q_0) \big\}$ and let
	\[
	\mathbb{D}_{Q_0}^{\rm small}:=\big\{Q\in\mathbb{D}_{Q_0}:\:\ell(Q)\le 2^{-k_0}\ell(Q_0)\big\},
	\quad
	\mathbb{D}_{Q_0}^{\rm big}:=\big\{Q\in\mathbb{D}_{Q_0}:\:\ell(Q)> 2^{-k_0}\ell(Q_0)\big\}.
	\]
	Note that
	\begin{equation}\label{tt-2}
	\iint_{T_{Q_0}}|\nabla u(X)|^2\delta(X)\,dX
	\leq\sum_{Q\in\mathbb{D}_{Q_0}^{\rm small}}\alpha_Q+\sum_{Q\in\mathbb{D}_{Q_0}^{\rm big}}\alpha_Q + \sum_{Q\in\mathcal{D}_0}\alpha_Q=:\mathrm{I}_{Q_0}+\mathrm{II}_{Q_0}.
	\end{equation}
	Note that if $Q\in\mathbb{D}_{Q_0}^{\rm small}$, there exists a unique $Q'\in\mathcal{D}_{0}$ such that $Q\in\mathbb{D}_{Q'}$, hence
	\begin{equation}\label{tt-3}
	\mathrm{I}_{Q_0}
	=
	\sum_{Q'\in\mathcal{D}_{0}}\sum_{Q\in\mathbb{D}_{Q'}}\alpha_Q
	=
	\sum_{Q'\in\mathcal{D}_{0}}\mathfrak{m}_\alpha(\mathbb{D}_{Q'})
	\leq
	\sum_{Q'\in\mathcal{D}_{0}}\|\mathfrak{m}_\alpha\|_{\mathcal{C}(Q')}\sigma(Q')
	\leq
	C_0\sigma(Q_0).
	\end{equation}
	where we have used our hypothesis since $\ell(Q')=2^{-k_0}\ell(Q_0)<\diam(\partial\Omega)/M_0$.
	For the second term, since $\delta(X)\approx\ell(Q)$ for $X\in U_Q$, we write
	\begin{multline}\label{tt-4}
	\mathrm{II}_{Q_0}
	\lesssim
	\sum_{Q\in\mathbb{D}_{Q_0}^{\rm big}}\ell(Q)\iint_{U_Q}|\nabla u(X)|^2\,dX
	\lesssim
	\sum_{Q\in\mathbb{D}_{Q_0}^{\rm big}}\ell(Q)^{-1}\iint_{U_Q^{*}}|u(X)|^2\,dX
	\\
	\lesssim 2^{k_0}\ell(Q_0)^{-1} |T_{Q_0}^*|
	\lesssim M_0\sigma(Q_0),
	\end{multline}
	where we have used Caccioppoli's inequality, the fact that the family $\{U_{Q}^*\}_{Q\in\mathbb{D}}$ has bounded overlap, the normalization $\|u\|_{L^{\infty}(\Omega)}=1$, \eqref{definicionkappa12}, the $\mathrm{AR}$ property, and that $2^{k_0}\lesssim M_0$. 
	Combining \eqref{tt-2}, \eqref{tt-3}, and \eqref{tt-4} we have that \eqref{dyadiccmeestimate} holds.

	Our next goal is to see that \eqref{dyadiccmeestimate} yields \eqref{cmeestimate:new}. For $x\in \partial\Omega$ and $0<r<\infty$. Set
	\[
	\mathcal{I}=\{I\in\mathcal W: I\cap B(x,r)\neq \emptyset\}.
	\]
	Given $I\in\mathcal{I}$, let $Z_I\in I\cap B(x,r)$ and note that by \eqref{constwhitney}
	\begin{equation}\label{25tw3}
	\diam(I)\le\dist(I,\partial\Omega)\le |Z_I-x|<r.
	\end{equation}
	Set
	\[
	\mathcal{I}^{\rm small}=\{I\in\mathcal I: \ell(I)<\diam(\partial\Omega)/4\},
	\qquad
	\mathcal{I}^{\rm big}=\{I\in\mathcal I: \ell(I)\ge \diam(\partial\Omega)/4\},
	\]
	with the understanding that $\mathcal{I}^{\rm big}=\emptyset$ if $\diam(\partial\Omega)=\infty$.
	Then,
	\begin{multline}\label{tt-6}
	\iint_{B(x,r)\cap\Omega}|\nabla u|^2\delta(X)\,dX
	\le
	\sum_{I\in \mathcal{I}^{\rm small}}\iint_{I}|\nabla u|^2\delta(X)\,dX
	\\
	+
	\sum_{I\in \mathcal{I}^{\rm big}}\iint_{I}|\nabla u|^2\delta(X)\,dX
	=
	\mathrm{I}+\mathrm{II},
	\end{multline}
	here we understand that $\mathrm{II}=0$ if $\mathcal{I}^{\rm big}=\emptyset$.
	
	To estimate $\mathrm{I}$ we set $r_0=\min\{r, \diam(\partial\Omega)/4\}$ and pick $k_2\in\mathbb{Z}$ so that $2^{k_2-1}\le r_0<2^{k_2}$. Set
	\[
	\mathcal{D}_1=\{Q\in\mathbb{D}: \ell(Q)=2^{k_2},\: Q\cap \Delta(x,3r)\neq\emptyset\}.
	\]
	Given $I\in  \mathcal{I}^{\rm small}$ we pick $y\in\partial\Omega$ so that $\dist(I,\partial\Omega)=\dist(I,y)$. Hence there exists a unique $Q_I\in\mathbb{D}$ so that
	$y\in Q_I$ and $\ell(Q_I)=\ell(I)<r_0\le \diam(\partial\Omega)/4$ by the definition of $\mathcal{I}^{\rm small}$ and our choice of $r_0$. This as mentioned above implies that $I\in \mathcal{W}_{Q_I}^*$. On the other hand by \eqref{25tw3}
	$$
	|y-x|
	\le
	\dist(y, I)+\diam(I)+|Z_I-x|
	<3r,
	$$
	hence there exists a unique $Q\in \mathcal{D}_1$ so that $y\in Q$. Since $\ell(Q_I)<r_0<2^{k_2}=\ell(Q)$  we conclude that $Q_I\subset Q$
	and consequently $I\subset \interior(U_{Q_I})\subset T_Q$. In short we have shown that  if $I\in  \mathcal{I}^{\rm small}$ then there exists $Q\in\mathcal{D}_1$ so that  $I\subset T_Q$. Thus,
	\begin{multline}\label{tt-7}
	\mathrm{I}
	\le
	\sum_{Q\in\mathcal{D}_1} \iint_{T_Q}|\nabla u|^2\delta\,dX
	\lesssim
	(C_0+M_0)\sum_{Q\in\mathcal{D}_1} \sigma(Q)
	=
	(C_0+M_0)\sigma\Big(\bigcup_{Q\in\mathcal{D}_1} Q\Big)
	\\
	\le
	(C_0+M_0)\sigma(\Delta(x,Cr))
	\lesssim
	(C_0+M_0) r^n,
	\end{multline}
	where we have used that the Whitney boxes have non-overlapping interiors,  \eqref{dyadiccmeestimate}, the fact that $\mathcal{D}_1$ is a pairwise disjoint family, that $\bigcup_{Q\in\mathcal{D}_1} Q\subset \Delta(x,Cr)$ ($C$ depends on $n$ and the AR constant), and that $\partial\Omega$ is Ahlfors regular.
	
	We now estimate $\mathrm{II}$ using \eqref{constwhitney}, Caccioppoli's inequality and our assumption $\|u\|_{L^{\infty}(\Omega)}=1$:
	\begin{multline}\label{tt-8}
	\mathrm{II}
	\lesssim
	\sum_{I\in \mathcal{I}^{\rm big}}\ell(I)\iint_{I}|\nabla u|^2\,dX
	\lesssim
	\sum_{I\in \mathcal{I}^{\rm big}}\ell(I)^{-1}\iint_{I^*}|u|^2\,dX
	\\
	\lesssim
	\sum_{I\in \mathcal{I}^{\rm big}} \ell(I)^n
	\leq
	\sum_{\frac{\diam(\partial\Omega)}4\le 2^k<r } 2^{kn} \#\{I\in \mathcal{I}^{\rm big}:\ell(I)=2^{k}\}.
	\end{multline}
	To estimate the last term we observe that if $Y\in I\in \mathcal{I}^{\rm big}$ we have by \eqref{constwhitney}
	\[
	|Y-x|
	\le
	\diam(I)+\dist(I,\partial\Omega)+\diam(\partial\Omega)\lesssim \ell(I).
	\]
	This and the fact that Whitney boxes have non-overlapping interiors imply
	\begin{multline*}
	\#\{I\in \mathcal{I}^{\rm big}:\ell(I)=2^{k}\}
	=
	2^{-k(n+1)}\sum_{I\in \mathcal{I}^{\rm big}:\ell(I)=2^{k}} |I|
	\\
	=
	2^{-k(n+1)}\Big|\bigcup_{I\in \mathcal{I}^{\rm big}:\ell(I)=2^{k}} I\Big|
	\le
	2^{-k(n+1)}
	|B(x, C2^k)|
	\lesssim
	1.
	\end{multline*}
	Therefore,
	\[
	\mathrm{II}
	\lesssim
	\sum_{\frac{\diam(\partial\Omega)}4\le 2^k<r } 2^{kn}
	\lesssim
	r^n.
	\]
	Collecting the estimates for $\mathrm{I}$ \eqref{tt-7} and $\mathrm{II}$ \eqref{tt-8} we obtain \eqref{cmeestimate:new}.
\end{proof}
\medskip

\begin{proof}[Proof of Theorem \ref{theor:cme-implies-ainf}: $(b)\Longrightarrow (a)$]
Let $u\in W^{1,2}_{\rm loc}(\Omega)\cap L^\infty(\Omega)$ be so that $Lu=0$ in the weak sense in $\Omega$. Our goal is to prove that \eqref{cmeestimate} holds. By homogeneity we may assume, without loss of generality, that $\|u\|_{L^\infty(\Omega)}=1$.  On the other hand, by Lemma \ref{lemmareduction} we can reduce matters to establish
that $\|\mathfrak{m}_\alpha\|_{\mathcal{C}(Q)}\leq C_0$, for every $Q\in\mathbb{D}(\partial\Omega)$ such that $\ell(Q)<\diam(\partial\Omega)/M_0$ and where $\alpha$ is given in \eqref{defcoefficients}. To show this we fix $M_0>2\kappa_0/c$, where $c$ is the corkscrew constant and $\kappa_0$ as in \eqref{definicionkappa12}. We also fix a cube $Q^0\in\mathbb{D}(\partial\Omega)$ with $\ell(Q^0)<\diam(\partial\Omega)/M_0$. Applying \cite[Lemma 3.12]{HMT-var} it suffices to show that for every $Q_0\in\mathbb{D}_{Q^0}$ we can find some pairwise disjoint family $\mathcal{F}_{Q_0}\subset\mathbb{D}_{Q_0}\setminus\{Q_0\}$ satisfying
\begin{equation}\label{amplecontact}
\sigma\Big(Q_0\setminus\bigcup_{Q_j\in\mathcal{F}_{Q_0}}Q_j\Big)\geq K_1^{-1}\sigma(Q_0),
\end{equation}
and prove that
\begin{equation}\label{sawtoothcarleson}
\mathfrak{m}_{\alpha}(\mathbb{D}_{\mathcal{F}_{Q_0},Q_0})\leq M_1\sigma(Q_0).
\end{equation}

With all the previous reductions our main goal is to find $\mathcal{F}_{Q_0}$ so that \eqref{amplecontact} holds and establish \eqref{sawtoothcarleson}. Having these in mind we let $B_{Q_0}:=B(x_{Q_0},r_{Q_0})$ with $r_{Q_0}\approx\ell(Q_0)$ as in \eqref{deltaQ}. Let $X_0:=X_{M_0\Delta_{Q_0}}$ be the corkscrew point relative to $M_0\Delta_{Q_0}$ (note that $M_0r_{Q_0}\leq M_0\ell(Q_0)<\diam(\partial\Omega)$). By our choice of $M_0$, it is clear that $Q_0\subset M_0\Delta_{Q_0}$ and also that $\delta(X_0)\geq c M_0 r_{Q_0}>2\kappa_0 r_{Q_0}$. Hence, by \eqref{definicionkappa12},
\begin{equation}\label{eq:X0-TQ}
X_0\in\Omega\setminus B_{Q_0}^*.
\end{equation}
On the other hand, $\delta(X_{Q_0})\approx \ell(Q_0)$, $\delta(X_0)\approx M_0\ell(Q_0)\geq\ell(Q_0)$, and $|X_0-X_{Q_0}|\lesssim M_0\ell(Q_0)$. Using Lemma \ref{bourgain} and Harnack's inequality, there exists $C_0\geq 1$ depending on the 1-sided $\mathrm{CAD}$ constants, the ellipticity of $L$, and on $M_0$ (which is already fixed), such that $\omega_{L}^{X_0}(Q_0)\geq C_0^{-1}$.

Next, we define the normalized elliptic measure and Green function as
\begin{equation}\label{normalizacion}
\omega_0:=C_0\,\sigma(Q_0)\omega_{L}^{X_0},\qquad\text{and}\qquad\mathcal{G}_0(\cdot):=C_0\,\sigma(Q_0)G_{L}(X_0,\cdot).
\end{equation}
Note the fact that $\omega_{L}^{X_0}(\partial\Omega)\leq 1$ implies
\[
1\leq\frac{\omega_0(Q_0)}{\sigma(Q_0)}\leq C_0.
\]
Recall that we have assumed that $\omega_{L}\in A_\infty(\partial\Omega)$ and, as observed above, this means after passing to the previous renormalization that $\omega_0\ll\sigma$ and we write $k_0=d\omega_0/d\sigma$ for the Radon-Nikodym derivative. Using \eqref{rhqL} we have that there exists $q>1$ such that since
 $Q_0\subset M_0\Delta_{Q_0}$, we have
$$
\bigg(\barint_{Q_0}k_0(y)^q\,d\sigma(y)\bigg)^{1/q}
\le
C_2.
$$
In particular, for any Borel set $F\subset Q_0$, using H\"older's inequality we obtain
$$
\frac{\omega_0(F)}{\sigma(Q_0)}\leq\bigg(\barint_{Q_0}\mathbf{1}_F(y)^{q'}\,d\sigma(y)\bigg)^{1/q'}\bigg(\barint_{Q_0}k_0(y)^q\,d\sigma(y)\bigg)^{1/q}
\leq C_2\Big(\frac{\sigma(F)}{\sigma(Q_0)}\Big)^{1/q'}.
$$
Hence we can apply \cite[Lemma 3.5]{HMT-var} to $\mu=\omega_0$, and extract a pairwise disjoint family $\mathcal{F}_{Q_0}=\{Q_j\}\subset\mathbb{D}_{Q_0}\setminus\{Q_0\}$ verifying \eqref{amplecontact}, as well as
\begin{equation}\label{propsawtooth2}
\frac{1}{2}\leq\frac{\omega_0(Q)}{\sigma(Q)}\leq K_0K_1,\qquad\forall\, Q\in\mathbb{D}_{\mathcal{F}_{Q_0},Q_0},
\end{equation}
with $K_1=(4K_0)^{1/\theta}$, $K_0=\max\{C_0,C_2\}$, and $\theta=1/q'$.

We next observe that if $I\in\mathcal{W}_Q^*$ with $Q\in\mathbb{D}_{\mathcal{F}_{Q_0},Q_0}$ then $2B_Q\subset B_{Q_0}^*$ (see \eqref{definicionkappa12}). Hence, using Harnack's inequality, parts $(b)$ and $(c)$ of Lemma \ref{proppde}, \eqref{propsawtooth2} and the $\mathrm{AR}$ property we have
\begin{equation}\label{G-delta}
\frac{\mathcal{G}_0(X_I)}{\ell(I)}
\approx
\frac{\mathcal{G}_0(X_I)}{\delta(X_I)}
\approx
\frac{\omega_0(\Delta_{Q})}{\sigma(Q)}
\approx
1,
\end{equation}
where $X_I$ is the center of $I$.

At this point, we are looking for $M_1$ independent of $Q_0$ and $Q^0$ such that \eqref{sawtoothcarleson} holds. Recalling \eqref{defcoefficients} we note that
\begin{multline}\label{cmegreen}
\mathfrak{m}_\alpha(\mathbb{D}_{\mathcal{F}_{Q_0},Q_0})
=
\sum_{Q\in\mathbb{D}_{\mathcal{F}_{Q_0},Q_0}}\iint_{U_Q}|\nabla u(X)|^2\delta(X)\,dX
\\
\approx
\sum_{Q\in\mathbb{D}_{\mathcal{F}_{Q_0},Q_0}}\iint_{U_Q}|\nabla u(X)|^2\mathcal{G}_0(X)\,dX
\lesssim
\iint_{\Omega_{\mathcal{F}_{Q_0},Q_0}}|\nabla u(X)|^2\mathcal{G}_0(X)\,dX,
\end{multline}
where we have used Harnack's inequality, \eqref{G-delta}, and the bounded overlap of the family $\{U_Q\}_{Q\in\mathbb{D}}$.

As in Section \ref{sec2.3} for every $N\ge 1$ we can consider the pairwise disjoint collection $\mathcal{F}_N:=\mathcal{F}_{Q_0}\big(2^{-N}\ell(Q_0)\big)$ which is the family of maximal cubes of the collection $\mathcal{F}_{Q_0}$ augmented by adding all of the cubes $Q\in\mathbb{D}_{Q_0}$ such that $\ell(Q)\leq 2^{-N}\ell(Q_0)$. In particular, $Q\in\mathbb{D}_{\mathcal{F}_N,Q_0}$ if and only if $Q\in\mathbb{D}_{\mathcal{F}_{Q_0},Q_0}$ and $\ell(Q)>2^{-N}\ell(Q_0)$. Clearly, $\mathbb{D}_{\mathcal{F}_N,Q_0}\subset\mathbb{D}_{\mathcal{F}_{N'},Q_0}$ if $N\leq N'$, and therefore $\Omega_{\mathcal{F}_N,Q_0}\subset\Omega_{\mathcal{F}_{N'},Q_0}\subset\Omega_{\mathcal{F}_{Q_0},Q_0}$. This and the monotone convergence theorem give that
\begin{equation}\label{TCM}
\iint_{\Omega_{\mathcal{F}_{Q_0},Q_0}}|\nabla u(X)|^2\mathcal{G}_0(X)\,dX
=
\lim_{N\rightarrow\infty}\iint_{\Omega_{\mathcal{F}_N,Q_0}}|\nabla u(X)|^2\mathcal{G}_0(X)\,dX.
\end{equation}

We now formulate an auxiliary result that will lead us to the desired estimate, namely \eqref{sawtoothcarleson}.

\begin{proposition}\label{pdelemma:I}
	Given $C_1\geq 1$, one can find $C$ such that if $\mathcal{F}_N\subset\mathbb{D}_{Q_0}$, $N\in\mathbb{N}$, is a family of pairwise disjoint dyadic cubes satisfying
	\begin{equation}\label{hyp-pdelemma:I}
	C_1^{-1}\leq\frac{\omega_0(Q)}{\sigma(Q)}\leq C_1\qquad\text{and}\qquad\ell(Q)>2^{-N}\ell(Q_0),\qquad\forall\, Q\in\mathbb{D}_{\mathcal{F}_N,Q_0},
	\end{equation}
	then
	\begin{equation}\label{conc-pdelemma:I}
	\iint_{\Omega_{\mathcal{F}_N,Q_0}}|\nabla u(X)|^2\mathcal{G}_0(X)\,dX\leq C\sigma(Q_0).
	\end{equation}
	Here, $C$ depends only on dimension, the 1-sided $\mathrm{CAD}$ constants, and the ellipticity of $L$.
\end{proposition}

Assuming this result momentarily, \eqref{propsawtooth2} and the construction of $\mathcal{F}_N$ give \eqref{hyp-pdelemma:I}. Next, we combine \eqref{cmegreen}, \eqref{TCM} and \eqref{conc-pdelemma:I} to conclude \eqref{sawtoothcarleson}. This completes the proof of $(b)\Longrightarrow (a)$ Theorem \ref{theor:cme-implies-ainf},
modulo obtaining the just stated proposition.
\end{proof}

\begin{proof}[Proof of Proposition \ref{pdelemma:I}]

	We introduce an adapted cut-off function which can be obtained from a straightforward modification of \cite[Lemma 4.44]{HMT-var} by simply replacing $\lambda$ by $2\lambda$ (recall that $\lambda$ appearing in Section \ref{sec2.3} can be chosen arbitrarily small).

\begin{lemma}\label{adaptedcutoff}
		There exists $\Psi_N\in C_c^{\infty}(\re^{n+1})$ such that
		\begin{list}{$(\theenumi)$}{\usecounter{enumi}\leftmargin=1cm \labelwidth=1cm \itemsep=0.1cm \topsep=.2cm \renewcommand{\theenumi}{\alph{enumi}}}
			\item $\mathbf{1}_{\Omega_{\mathcal{F}_N,Q_0}}\lesssim\Psi_N\leq\mathbf{1}_{\Omega^{*}_{\mathcal{F}_N,Q_0}}$.
			\item $\sup_{X\in\Omega}|\nabla\Psi_N(X)|\delta(X)\lesssim 1$.
			\item Set
			$$
			\mathcal{W}_N:=\bigcup_{Q\in\mathbb{D}_{\mathcal{F}_N,Q_0}}\mathcal{W}_Q^*,\qquad
			\mathcal{W}_N^\Sigma:=\big\{I\in\mathcal{W}_N:\:\exists J\in\mathcal{W}\setminus\mathcal{W}_N\quad\text{with}\quad\partial I\cap\partial J\neq\emptyset\big\}.
			$$
			Then
			\begin{equation}\label{tt-9}
			\nabla\Psi_N\equiv 0\quad\text{in}\quad\bigcup_{I\in\mathcal{W}_N\setminus\mathcal{W}_N^\Sigma}I^{**}\qquad\text{and}\qquad\sum_{I\in\mathcal{W}_N^{\Sigma}}\ell(I)^n\lesssim\sigma(Q_0),
			\end{equation}
			with implicit constants depending only on the allowable parameters but uniform in $N$.
		\end{list}
	\end{lemma}
	\medskip

Taking then $\Psi_N$ as above, Leibniz's rule leads us to
\begin{multline}\label{leibniz:I}
A\nabla u\cdot\nabla u\,\mathcal{G}_0\,\Psi_N^2
=
A\nabla u\cdot\nabla(u\,\mathcal{G}_0\,\Psi_N^2)
-\tfrac{1}{2}A\nabla(u^2\,\Psi_N^2)\cdot\nabla\,\mathcal{G}_0
\\
+\tfrac{1}{2}A\nabla(\Psi_N^2)\cdot\nabla\mathcal{G}_0\,u^2-\tfrac{1}{2}A\nabla (u^2)\cdot\nabla(\Psi_N^2)\,\mathcal{G}_0.
\end{multline}
	
	Note that $u\,\mathcal{G}_0\,\Psi_N^2\in W^{1,2}_0(\Omega^{**}_{\mathcal{F}_N,Q_0})$ since $\overline{\Omega^{**}_{\mathcal{F}_N,Q_0}}$ is a compact subset of $\Omega$ (indeed by construction $\dist(\overline{\Omega^{**}_{\mathcal{F}_N,Q_0}},\partial\Omega)\gtrsim 2^{-N}\ell(Q_0)$), $u\in W^{1,2}_{\rm loc}(\Omega)\cap L^\infty(\Omega)$, $\mathcal{G}_0\in W^{1,2}_{\rm loc}(\Omega\setminus \{X_0\})$, $\overline{\Omega^{**}_{\mathcal{F}_N,Q_0}}\subset\overline{T_{Q_0}^{**}}\subset \frac12 B_{Q_0}^*$ (cf. \eqref{definicionkappa12}), and \eqref{eq:X0-TQ}. Moreover, since $u\in W^{1,2}_{\rm loc}(\Omega)$ it follows that $u\in  W^{1,2}(\Omega^{**}_{\mathcal{F}_N,Q_0})$. All these plus the fact that $L u=0$ in the weak sense in $\Omega$ easily give
	\begin{equation}\label{vanish:I}
	\iint_{\Omega}A\nabla u\cdot\nabla(u\,\mathcal{G}_0\Psi_N^2)\,dX
	=
	\iint_{\Omega^{**}_{\mathcal{F}_N,Q_0}}A\nabla u\cdot\nabla(u\,\mathcal{G}_0\Psi_N^2)\,dX
	=
	0.
	\end{equation}

	Moreover as above $u^2\,\Psi_N^2\in W^{1,2}_0(\Omega^{**}_{\mathcal{F}_N,Q_0})$. Also, Lemma \ref{lemagreen} (see in particular \eqref{greenpole}) gives at once that $\mathcal{G}_0\in W^{1,2}(\Omega^{**}_{\mathcal{F}_N,Q_0})$ and $L^\top\mathcal{G}_0=0$ in the weak sense in $\Omega\setminus\{X_0\}$. Thus, we easily obtain
	\begin{equation}\label{vanish2:I}
	\iint_{\Omega}A\nabla(u^2\,\Psi_N^2)\cdot\nabla\mathcal{G}_0\,dX
	=
	\iint_{\Omega^{**}_{\mathcal{F}_N,Q_0}} A^\top \nabla\mathcal{G}_0\cdot\nabla(u^2\,\Psi_N^2)\,dX
	=
	0.
	\end{equation}

	Using ellipticity, \eqref{leibniz:I}, \eqref{vanish:I}, \eqref{vanish2:I}, the fact that $\|u\|_{L^{\infty}(\Omega)}=1$, and Lemma \ref{adaptedcutoff}, we have
	\begin{multline}\label{firstbound:I}
	\iint_{\Omega_{\mathcal{F}_N,Q_0}}|\nabla u|^2\,\mathcal{G}_0\,dX\leq
	\iint_{\Omega}|\nabla u|^2\,\mathcal{G}_0\,\Psi_N^2\,dX
	\lesssim
	\iint _{\Omega}A\nabla u\cdot\nabla u\,\mathcal{G}_0\,\Psi_N^2\,dX
	\\
	\lesssim
	\iint_{\Omega}\Big(|\nabla\mathcal{G}_0|+|\nabla u|\,\mathcal{G}_0\Big)\,|\nabla\Psi_N|\,\Psi_N\,dX
	=:
	\mathrm{I}.
	\end{multline}
	To estimate $\mathrm{I}$ we use Lemma \ref{adaptedcutoff}, Caccioppoli's and Harnack's inequalities, and the fact that $\|u\|_{L^{\infty}(\Omega)}=1$:
\begin{equation}\label{acotI:1}
	\mathrm{I}\, 
	\lesssim
	\,\sum_{I\in\mathcal{W}_N^\Sigma}\ell(I)^{-1}\bigg(\iint_{I^{**}}|\nabla\mathcal{G}_0|\,dX+\iint_{I^{**}}|\nabla u|\,\mathcal{G}_0\,dX\bigg)\\
	\lesssim
	\sum_{I\in\mathcal{W}_N^\Sigma}\ell(I)^{n-1}\mathcal{G}_0(X_I),
\end{equation}
	where $X_I$ is the center of $I$. Note that for every $I\in\mathcal{W}_N^\Sigma$ there is $Q\in\mathbb{D}_{\mathcal{F}_N,Q_0}$ such that $I\in\mathcal{W}_Q^*$. Hence we can use \eqref{G-delta} and \eqref{tt-9} to obtain
	\begin{equation}\label{acot-I}
	\mathrm{I}\,\lesssim\,\sum_{I\in\mathcal{W}_N^\Sigma}\ell(I)^{n-1}\mathcal{G}_0(X_I)\lesssim\sum_{I\in\mathcal{W}_N^\Sigma} \ell(I)^n
	\lesssim\sigma(Q_0).
	\end{equation}
	Plugging \eqref{acot-I} into \eqref{firstbound:I} we get \eqref{conc-pdelemma:I} and complete the proof of Lemma \ref{pdelemma:I}.
\end{proof}

\section{Proof of Theorems \ref{theor:perturbation-ainf-cme} and  \ref{theor:perturbation-ainf-cme-A-At}}\label{section:Proof-Perturb}

We will prove Theorems \ref{theor:perturbation-ainf-cme} and  \ref{theor:perturbation-ainf-cme-A-At} by showing that all bounded weak solutions satisfy the Carleson measure estimate \eqref{cmeestimate}, in which case Theorem \ref{theor:cme-implies-ainf} will give the $A_\infty$ properties. First we prove an integration by parts identity.

\begin{lemma}\label{int-parts}
	Let $D=(d_{i,j}\big)_{i,j=1}^{n+1}\in L^{\infty}(\Omega)\cap {\rm Lip}_{\rm loc}(\Omega) $ be an antisymmetric real matrix and set for $X\in\Omega$
\begin{equation}\label{def-B0:gral}
\div_C D(X):=
\big(\div\big(d_{\cdot,j}(X))\big)_{1\le j\le n+1}
=
\bigg(\sum_{i=1}^{n+1}\partial_{i} d_{i,j}(X)\bigg)_{1\leq j\leq n+1},
\end{equation}
which is the vector formed by taking the divergence operator acting on the columns of $D$.
Then,
	\begin{equation}\label{int-parts-equal}
	\iint_{\Omega}D(X)\nabla u(X)\cdot\nabla v(X)\, dX
	=
	-\iint_{\Omega}\div_C D(X)\cdot\nabla u(X)\,v(X)\,dX,
	\end{equation}
for every $u\in W^{1,2}_{\rm loc}(\Omega)$ and every $v\in W^{1,2}(\Omega)$ such that $K=\supp(v)\subset\Omega$ is compact.
\end{lemma}
\begin{proof}
We first consider the case $u,v\in C_c^{\infty}(\Omega)$. Using Leibniz's rule and the fact that $D$ is antisymmetric we have that
\[
\div(D\nabla u)
=
\sum_{i=1}^{n+1}\sum_{j=1}^{n+1}\partial_i d_{i,j}\partial_j u+\sum_{i=1}^{n+1}\sum_{j=1}^{n+1} d_{i,j}\partial_i\partial_j u
=
\div_C D\cdot \nabla u.
\]
Using this we integrate by parts to obtain
\[
\iint_{\Omega}D \nabla u\cdot\nabla v\, dX
=
-\iint_{\Omega} \div(D\nabla u)\, v\, dX
=
-\iint_{\Omega}\div_C D \cdot\nabla u\,v\,dX.
\]
	
	To obtain the general case let $u\in W^{1,2}_{\rm loc}(\Omega)$ and $v\in W^{1,2}(\Omega)$ such that $K=\supp(v)\subset\Omega$ is compact. It is standard to see, using for instance the Whitney covering, that we can find $\Phi_K\in C_c^\infty(\Omega)$ so that $\Phi_K\equiv 1$ in $K$. Write $K^*=\supp(\Phi_K)$ which is a compact subset of $\Omega$ and define
	\[
	U:=\{X\in\Omega: \dist(X,K^*)<\dist(K^*,\partial\Omega)/2\}
	\]
	which satisfies $\dist(\overline{U},\partial\Omega)\ge\dist(K^*,\partial\Omega)/2>0$, hence $\overline{U}$ it is also a compact subset of $\Omega$. Since $u\in W^{1,2}_{\rm loc}(\Omega)$ we clearly have that $u\Phi_K\in W^{1,2}_0(U)$ and hence we can find $\{u_j\}_j\subset C^\infty_c(U)$ so that $u_j\to u\Phi_K$ in $W^{1,2}(U)$. Also, since  $v\in W^{1,2}(\Omega)$ verifies $K=\supp(v)\subset\Omega$ it is also easy to see that $v\in W^{1,2}_0(U)$ and hence we can find $\{v_j\}_j\subset C^\infty_c(U)$ so that $v_j\to v$ in $W^{1,2}(U)$. Notice that extending the $u_j$'s and $v_j$'s as $0$ outside of $U$ one sees that $\{u_j\}_j, \{v_j\}_j\subset C^\infty_c(\Omega)$. Thus, we can use \eqref{int-parts-equal}  and for every $j$
	\begin{equation}\label{conc-smooth:j}
	\iint_{\Omega}D\nabla u_j\cdot\nabla v_j\, dX
	=
-	\iint_{\Omega}\div_C D \cdot\nabla u_j\,v_j\,dX.
	\end{equation}
	Note that using that $\supp(v_j),\supp(v)=K\subset U$ and that $\Phi_K\equiv 1$ in $K\subset U$ we have
	\begin{align}\label{tt-3A}
	&\Big|\iint_{\Omega} D\nabla u\cdot\nabla v\, dX
	-
	\iint_{\Omega}D\nabla u_j\cdot\nabla v_j\, dX\Big|
	\\
	&\qquad=
	\Big|\iint_{\Omega} D\nabla (u\Phi_K)\cdot\nabla v\, dX
	-
	\iint_{\Omega}D\nabla u_j\cdot\nabla v_j\, dX\Big|
	\\
	&\qquad\le
	\|D\|_{L^\infty(\Omega)}
	\big(
	\|\nabla  (u\Phi_K)\|_{L^2(U)}\|\nabla v-\nabla v_j\|_{L^2(U)}
	\\
	&\hskip4cm	
	+\|\nabla (u\Phi_K)-\nabla u_j\|_{L^2(U)}\|\nabla v_j\|_{L^2(U)}\big),
	\end{align}
	and the last term converges to $0$ as $j\to\infty$ since $D\in L^\infty(\Omega)$ and the $v_j$'s are uniformly bounded in $W^{1,2}(U)$. Analogously,
	\begin{align}\label{tt-4A}
	&\Big|
	\iint_{\Omega} \div_C D \cdot\nabla u\,v\,dX-
	\iint_{\Omega} \div_C D \cdot\nabla u_j\,v_j\,dX\Big|
	\\
	&\qquad\qquad
	=
	\Big|
	\iint_{\Omega} \div_C D \cdot\nabla (u\Phi_K))\,v\,dX-
	\iint_{\Omega}\div_C D \cdot\nabla u_j\,v_j\,dX\Big|
	\\
	&\qquad\qquad\le
	\|\nabla D\|_{L^\infty(U)}
	\big(
	\|\nabla (u\Phi_K)\|_{L^2(U)}\|v-v_j\|_{L^2(U)}
	\\
	&\hskip4cm
	+\|\nabla (u\Phi_K)-\nabla u_j\|_{L^2(U)}\|v_j\|_{L^2(U)}\big),
	\end{align}
	which also converges to 0 as $j\to\infty$ since  $D\in{\rm Lip}_{\rm loc}(\Omega)$ and the $v_j$'s are uniformly bounded in $W^{1,2}(U)$. Combining \eqref{tt-3A}, \eqref{tt-4A} and \eqref{conc-smooth:j} yields \eqref{int-parts-equal}.
\end{proof}
\medskip

We show that Theorems \ref{theor:perturbation-ainf-cme} and  \ref{theor:perturbation-ainf-cme-A-At} follow from the following more general result which is interesting on its own right:

\begin{theorem}\label{theor:perturbation-general}
	Let $\Omega\subset\re^{n+1}$, $n\ge 2$, be a 1-sided $\mathrm{CAD}$ (cf. Definition \ref{def-1CAD}). Let $L_1u=-\div(A_1\nabla u)$ and $L_0u=-\div(A_0\nabla u)$ be real (not necessarily symmetric) elliptic operators (cf. Definition \ref{ellipticoperator}). Suppose that $A_0-A_1=A+D$ where $A, D\in L^\infty(\Omega)$ are real matrices satisfying the following conditions:
\begin{list}{$(\theenumi)$}{\usecounter{enumi}\leftmargin=.8cm
		\labelwidth=.8cm\itemsep=0.2cm\topsep=.1cm
		\renewcommand{\theenumi}{\roman{enumi}}}

\item Define for $X\in\Omega$
\begin{equation}\label{discrepancia-gral}  
a(X):=\sup_{Y\in B(X,\delta(X)/2)}|A(Y)|,
\end{equation}
where $\delta(X):=\dist(X,\partial\Omega)$,  and assume that it satisfies the Carleson measure condition
\begin{equation}\label{eq:defi-vertiii:gral}
C_A:=\sup_{\substack{x\in\partial\Omega \\ 0<r<\diam(\partial\Omega)}}\frac{1}{\sigma(B(x,r)\cap\partial\Omega)}\iint_{B(x,r)\cap\Omega}\frac{a(X)^2}{\delta(X)}\,dX<\infty.
\end{equation}	

\item $D\in {\rm Lip}_{\rm loc}(\Omega)$ is antisymmetric and suppose that $\div_C D$ defined in \eqref{def-B0:gral} satisfies the Carleson measure condition
\begin{equation}\label{carleson-B0-gral}
C_D :=\sup_{\substack{x\in\partial\Omega \\ 0<r<\diam(\partial\Omega)}} \frac{1}{\sigma(B(x,r)\cap\partial\Omega)}
\iint_{B(x,r)\cap\Omega}\big|\div_C D(X)\big|^2\delta(X)\,dX<\infty.
\end{equation}
\end{list}	
Then,  $\omega_{L_0}\in A_\infty(\partial\Omega)$ if and only if $\omega_{L_1}\in A_\infty(\partial\Omega)$ (cf. Definition \ref{defi:Ainfty}).  	 	
\end{theorem}

Assuming this result we can easily prove Theorems \ref{theor:perturbation-ainf-cme} and  \ref{theor:perturbation-ainf-cme-A-At}:

\begin{proof}[Proof of Theorem \ref{theor:perturbation-ainf-cme}]
For $L_0$ and $L_1$ as in the statement of Theorem \ref{theor:perturbation-ainf-cme}  we set $A=A_0-A_1$ and $D=0$. Thus, it suffices to check that $A$ and $D$ satisfy the required conditions in Theorem \ref{theor:perturbation-general}. For $(i)$ notice that
$a=\varrho(A_1, A_0)$ (cf. \eqref{discrepancia-gral} and \eqref{discrepancia}), hence \eqref{eq:defi-vertiii} gives immediately \eqref{eq:defi-vertiii:gral}. On the other hand since $D=0$ we clearly have all the conditions in $(ii)$. With all these in hand, Theorem \ref{theor:perturbation-general} gives at once the desired conclusion.
\end{proof}

\begin{proof}[Proof of Theorem \ref{theor:perturbation-ainf-cme-A-At}]
Set $A_0=A$, $A_1=A^\top$, $\widetilde{A}=0$ and $D=A-A^\top$ so that $A_0-A_1=\widetilde{A}+D$. As before we can easily see that $\widetilde{A}$ and $D$ satisfy the required conditions in Theorem \ref{theor:perturbation-general}. This time $(i)$ is trivial. For $(ii)$ notice that by assumption $D=A-A^\top\in{\rm Lip}_{\rm loc}(\Omega)$ and also that \eqref{carleson-B0} yields \eqref{carleson-B0-gral} since \eqref{def-B0} agrees with \eqref{def-B0:gral}. As a result, we can invoke Theorem \ref{theor:perturbation-general} obtaining that $\omega_{L}\in A_\infty(\partial\Omega)$ if and only if $\omega_{L^\top}\in A_\infty(\partial\Omega)$. 

On the other hand, if we let $A_0=A$, $A_1=A^{\rm sym}=\frac{A+A^\top}2$, $\widetilde{A}=0$ and $D=\frac{A-A^\top}2$ so that $A_0-A_1=\widetilde{A}+D$, the same argument yields that $\omega_{L}\in A_\infty(\partial\Omega)$ if and only if $\omega_{L^{\rm sym}}\in A_\infty(\partial\Omega)$.

\end{proof}

Besides the previous results one can easily get other interesting perturbation results from Theorem \ref{theor:perturbation-general}. For instance suppose that $L_0 u=-\div(A_0\nabla u)$ has an associated elliptic measure satisfying $\omega_{L_0}\in A_\infty(\partial\Omega)$. Let $D$ be a real antisymmetric matrix
with locally Lipschitz coefficients and assume that $\|D\|_{L^\infty(\Omega)}<\lambda_0 $ where $\lambda_0>0$ is so that $A(X)\xi\cdot\xi\ge \lambda_0\,|\xi|^2$ for all $\xi\in\re^{n+1}$ and a.e.~$X\in\Omega$. The latter ensures that $A_1=A_0+D$ is uniformly elliptic and hence if we assume that $\div_C D$ satisfies \eqref{carleson-B0-gral} then Theorem \ref{theor:perturbation-general} gives immediately that $\omega_{L_1}\in A_\infty(\partial\Omega)$ where $L_1 u=-\div(A_1\nabla u)$. In particular, the $A_\infty$ property is preserved under perturbations by antisymmetric ``sufficiently small'' matrices $D$ with locally Lipschitz coefficients so that $|\nabla D|^2\delta$ satisfies a Carleson measure condition.

\begin{proof}[Proof of Theorem \ref{theor:perturbation-general}]
By symmetry it suffices to assume that $\omega_{L_0}\in A_\infty(\partial\Omega)$ and prove that $\omega_{L_1}\in A_\infty(\partial\Omega)$. By Theorem \ref{theor:cme-implies-ainf} it suffices to show that given $u\in W^{1,2}_{\rm loc}(\Omega)\cap L^\infty(\Omega)$ with $L_1u=0$ in the weak sense in $\Omega$ then \eqref{cmeestimate} holds. As before, by homogeneity we may assume without loss of generality that $\|u\|_{L^\infty(\Omega)}=1$.  We can now follow closely the proof of $(b)\Longrightarrow (a)$ in Theorem \ref{theor:cme-implies-ainf} with the following changes. Here we are assuming that $\omega_{L_0}\in A_\infty(\partial\Omega)$ and hence \eqref{normalizacion} needs to be replaced by
\begin{equation}\label{normalizacion:II}
\omega_0:=C_0\,\sigma(Q_0)\omega_{L_0}^{X_0},\qquad\text{and}\qquad\mathcal{G}_0(\cdot):=C_0\,\sigma(Q_0)G_{L_0}(X_0,\cdot),
\end{equation}
where $X_0:=X_{M_0\Delta_{Q_0}}$ is chosen as before so that \eqref{eq:X0-TQ} holds.

Notice that in the present situation $u$ satisfies $L_1 u=0$ (as opposed to what happened above where both $u$ and $\mathcal{G}_0$ where associated with the same operator).
Other than that, and keeping in mind \eqref{normalizacion:II}, all estimates \eqref{propsawtooth2}--\eqref{TCM} hold.
Thus it is straightforward to see that everything reduces to obtain the following analog of Proposition \ref{pdelemma:I}:

\begin{proposition}\label{pdelemma:II}
Given $C_1\geq 1$, one can find $C$ such that if $\mathcal{F}_N\subset\mathbb{D}_{Q_0}$, $N\in\mathbb{N}$, is a family of pairwise disjoint dyadic cubes satisfying
\begin{equation}\label{hyp-pdelemma}
C_1^{-1}\leq\frac{\omega_0(Q)}{\sigma(Q)}\leq C_1\qquad\text{and}\qquad\ell(Q)>2^{-N}\ell(Q_0),\qquad\forall\, Q\in\mathbb{D}_{\mathcal{F}_N,Q_0},
\end{equation}
then
\begin{equation}\label{conc-pdelemma}
\iint_{\Omega_{\mathcal{F}_N,Q_0}}|\nabla u(X)|^2\mathcal{G}_0(X)\,dX\leq C\sigma(Q_0).
\end{equation}
Here, $C$ depends only on dimension, the 1-sided $\mathrm{CAD}$ constants, the ellipticity of $L_0$ and $L_1$, and on $C_A$ and $C_D$.
\end{proposition}

The proof of Theorem \ref{theor:perturbation-general} follows from Proposition \ref{pdelemma:II} as the proof in section \ref{section:Proof-CME:b->a} follows from Proposition \ref{pdelemma:I}.
\end{proof}

\begin{proof}[Proof of Proposition \ref{pdelemma:II}]
Take $\Psi_N$ from Lemma \ref{adaptedcutoff} and write $\mathcal{E}(X):=A_1(X)-A_0(X)$. Then Leibniz's rule leads us to
\begin{multline}\label{leibniz}
A_1\nabla u\cdot\nabla u\,\mathcal{G}_0\,\Psi_N^2
=
A_1\nabla u\cdot\nabla(u\,\mathcal{G}_0\,\Psi_N^2)
-\tfrac{1}{2}A_0\nabla(u^2\,\Psi_N^2)\cdot\nabla\,\mathcal{G}_0
\\
+\tfrac{1}{2}A_0\nabla(\Psi_N^2)\cdot\nabla\mathcal{G}_0\,u^2-\tfrac{1}{2}A_0\nabla (u^2)\cdot\nabla(\Psi_N^2)\,\mathcal{G}_0
-\tfrac{1}{2}\mathcal{E}\,\nabla (u^2)\cdot\nabla(\mathcal{G}_0\,\Psi_N^2).
\end{multline}

Note that $u\,\mathcal{G}_0\,\Psi_N^2\in W^{1,2}_0(\Omega^{**}_{\mathcal{F}_N,Q_0})$ since $\overline{\Omega^{**}_{\mathcal{F}_N,Q_0}}$ is a compact subset of $\Omega$ (indeed by construction $\dist(\overline{\Omega^{**}_{\mathcal{F}_N,Q_0}},\partial\Omega)\gtrsim 2^{-N}\ell(Q_0)$), $u\in W^{1,2}_{\rm loc}(\Omega)\cap L^\infty(\Omega)$, $\mathcal{G}_0\in W^{1,2}_{\rm loc}(\Omega\setminus \{X_0\})$, $\overline{\Omega^{**}_{\mathcal{F}_N,Q_0}}\subset\overline{T_{Q_0}^{**}}\subset \frac12 B_{Q_0}^*$ (cf. \eqref{definicionkappa12}), and \eqref{eq:X0-TQ}. Moreover, since $u\in W^{1,2}_{\rm loc}(\Omega)$ it follows that $u\in  W^{1,2}(\Omega^{**}_{\mathcal{F}_N,Q_0})$.  
Thus since $L_1 u=0$ in the weak sense in $\Omega$ we have
\begin{equation}\label{vanish}
\iint_{\Omega}A_1\nabla u\cdot\nabla(u\,\mathcal{G}_0\Psi_N^2)\,dX
=
\iint_{\Omega^{**}_{\mathcal{F}_N,Q_0}}A_1\nabla u\cdot\nabla(u\,\mathcal{G}_0\Psi_N^2)\,dX
=
0.
\end{equation}

On the other hand, much as before $u^2\,\Psi_N^2\in W^{1,2}_0(\Omega^{**}_{\mathcal{F}_N,Q_0})$. Also, Lemma \ref{lemagreen} (see in particular \eqref{greenpole}) gives at once that $\mathcal{G}_0\in W^{1,2}(\Omega^{**}_{\mathcal{F}_N,Q_0})$ and $L_0^\top\mathcal{G}_0=0$ in the weak sense in $\Omega\setminus\{X_0\}$. Thus, we easily obtain
\begin{equation}\label{vanish2}
\iint_{\Omega}A_0\nabla(u^2\,\Psi_N^2)\cdot\nabla\mathcal{G}_0\,dX
=
\iint_{\Omega^{**}_{\mathcal{F}_N,Q_0}} A_0^\top \nabla\mathcal{G}_0\cdot\nabla(u^2\,\Psi_N^2)\,dX
=
0.
\end{equation}

Using ellipticity, \eqref{leibniz}, \eqref{vanish}, \eqref{vanish2}, the fact that $\|u\|_{L^{\infty}(\Omega)}=1$, and Lemma \ref{adaptedcutoff}, we have
\begin{multline}\label{firstbound}
\iint_{\Omega}|\nabla u|^2\,\mathcal{G}_0\,\Psi_N^2\,dX
\lesssim
\iint _{\Omega}A_1\nabla u\cdot\nabla u\,\mathcal{G}_0\,\Psi_N^2\,dX
\\
\lesssim
\iint_{\Omega}\Big(|\nabla\mathcal{G}_0|+|\nabla u|\,\mathcal{G}_0\Big)\,|\nabla\Psi_N|\,\Psi_N\,dX
+
\bigg|\iint_{\Omega}\mathcal{E}\nabla (u^2)\cdot\nabla(\mathcal{G}_0\,\Psi_N^2)\,dX\bigg|
=:
\mathrm{I}+\mathrm{II}.
\end{multline}

Much as in \eqref{acotI:1} and  \eqref{acot-I} we can show that $
\mathrm{I}\lesssim\sigma(Q_0)$.
To estimate $\mathrm{II}$ note that since $\mathcal{E}=A_1-A_0=-(A+D)$ it follows that
\begin{multline}\label{acot-II}
\mathrm{II}
\le
\bigg|\iint_{\Omega} A\nabla (u^2)\cdot\nabla(\mathcal{G}_0\,\Psi_N^2)\,dX\bigg|
+
\bigg|\iint_{\Omega} D\nabla (u^2)\cdot\nabla(\mathcal{G}_0\,\Psi_N^2)\,dX\bigg|
=
\mathrm{II}_1+\mathrm{II}_2.
\end{multline}
For the term $\mathrm{II}_1$ we use that $A\in L^\infty(\Omega)$ and the fact that $\|u\|_{L^{\infty}(\Omega)}=1$ to obtain
\begin{equation}\label{acot-II1}
\mathrm{II}_1\lesssim\iint_{\Omega}|A|\,|\nabla u|\,|\nabla\mathcal{G}_0|\,\Psi_N^2\,dX+\iint_{\Omega}|\nabla (u^2)|\,|\nabla(\Psi_N^2)|\,\mathcal{G}_0\,dX=:\mathrm{III}_1+\mathrm{III}_2.
\end{equation}
For $\mathrm{III}_1$ we note that $\sup_{I^{**}}|A|\leq\inf_{I^*}a$ for every $I\in\mathcal{W}$, since $I^{**}\subset\{Y\in\Omega:\,|Y-X|<\delta(X)/2\}$ for every $X\in I^{*}$ (see \eqref{constwhitney}). Hence,
Lemma \ref{adaptedcutoff}, Caccioppoli's and Harnack's inequalities, \eqref{G-delta},  the fact that the family $\{I^{**}\}_{I\in\mathcal{W}}$ has bounded overlap,
and \eqref{definicionkappa12} yield
\begin{align}\label{acot-III1}
\mathrm{III}_1
&\lesssim\,
\sum_{I\in\mathcal{W}_N}\sup_{I^{**}}|A|\bigg(\iint_{I^{**}}|\nabla u|^2\,\Psi_N^2\,dX\bigg)^{\frac12}\bigg(\iint_{I^{**}}|\nabla \mathcal{G}_0|^2\,dX\bigg)^{\frac12}
\\
&\lesssim\,
\sum_{I\in\mathcal{W}_N}\bigg(\iint_{I^{**}}|\nabla u|^2\,\Psi_N^2\,dX\bigg)^{\frac12}\Big(\sup_{I^{**}}|A|^2\,\mathcal{G}_0(X_I)^2\,\ell(I)^{n-1}\Big)^{\frac12}\nonumber
\\
&\lesssim\,
\sum_{I\in\mathcal{W}_N}\bigg(\iint_{I^{**}}|\nabla u|^2\,\mathcal{G}_0\,\Psi_N^2\,dX\bigg)^{\frac12}\bigg(\iint_{I^{*}}\frac{a(X)^2}{\delta(X)}\,dX\bigg)^{\frac12}\nonumber
\\
&\lesssim\,
\bigg(\iint_{\Omega}|\nabla u|^2\,\mathcal{G}_0\,\Psi_N^2\,dX\bigg)^{\frac12}\bigg(\iint_{B_{Q_0}^*}\frac{a(X)^2}{\delta(X)}\,dX\bigg)^{\frac12} \nonumber
\\
&\lesssim\,C_A^{\frac12}\,
\bigg(\iint_{\Omega}|\nabla u|^2\,\mathcal{G}_0\,\Psi_N^2\,dX\bigg)^{\frac12}\sigma(Q_0)^{\frac12},\nonumber
\end{align}
where in the last estimate we have used \eqref{eq:defi-vertiii:gral} and AR along with the fact that $r(B_{Q_0}^*)=2\kappa_0r_{Q_0}\le 2\kappa_0\ell(Q_0)\le 2\kappa_0\diam(\partial\Omega)/M_0<\diam(\partial\Omega)$ by our choice of $M_0$.
On the other hand, we observe that
\begin{align}\label{acot-III2}
\mathrm{III}_2
&\lesssim
\iint_{\Omega}|\nabla u|\,|\nabla\Psi_N|\,\mathcal{G}_0\,\Psi_N\,dX
\\
&\lesssim
\bigg(\iint_{\Omega}|\nabla u|^2\,\mathcal{G}_0\,\Psi_N^2\,dX\bigg)^{\frac12}\bigg(\iint_{\Omega}|\nabla \Psi_N|^2\,\mathcal{G}_0\,dX\bigg)^{\frac12}
\nonumber
\\
&\lesssim
\bigg(\iint_{\Omega}|\nabla u|^2\,\mathcal{G}_0\,\Psi_N^2\,dX\bigg)^{\frac12}\bigg(\sum_{I\in\mathcal{W}_N^\Sigma}\ell(I)^{n-1}\mathcal{G}_0(X_I)\bigg)^{\frac12}
\nonumber
\\
&\lesssim
\bigg(\iint_{\Omega}|\nabla u|^2\,\mathcal{G}_0\,\Psi_N^2\,dX\bigg)^{\frac12}\sigma(Q_0)^{\frac12},
\nonumber
\end{align}
where we have used Lemma \ref{adaptedcutoff}, Harnack's inequality, the normalization $\|u\|_{L^{\infty}(\Omega)}=1$ and the last estimate follows as in \eqref{acot-I}.

Let us now turn our attention to estimating $\mathrm{II}_2$. Note that $u^2\in W^{1,2}_{\rm loc}(\Omega)$ since $u\in W^{1,2}_{\rm loc}(\Omega)\cap L^\infty(\Omega)$;
$\supp(\mathcal{G}_0\,\Psi_N^2)\subset \overline{\Omega^{*}_{\mathcal{F}_N,Q_0}}$ which is a compact subset of $\Omega$ since by construction $\dist(\overline{\Omega^{*}_{\mathcal{F}_N,Q_0}},\partial\Omega)\gtrsim 2^{-N}\ell(Q_0)$; and  finally
$\mathcal{G}_0\,\Psi_N^2\in W^{1,2}(\Omega)$ since $\mathcal{G}_0\in W^{1,2}_{\rm loc}(\Omega\setminus \{X_0\})$, $\overline{\Omega^{*}_{\mathcal{F}_N,Q_0}}\subset\overline{T_{Q_0}^{*}}\subset \frac12 B_{Q_0}^*$ (cf. \eqref{definicionkappa12}), and \eqref{eq:X0-TQ}. Thus we can invoke Lemma \ref{int-parts} to see that
\begin{align}\label{acot-II2}
\mathrm{II}_2
&
=
\bigg|\iint_{\Omega} \div_C D \cdot \nabla (u^2)\,\mathcal{G}_0\,\Psi_N^2\,dX\bigg|
\\ \nonumber
&\lesssim
\bigg(\iint_{\Omega}|\nabla u|^2\,\mathcal{G}_0\,\Psi_N^2\,dX\Bigg)^{\frac12}\bigg(\iint_{\Omega}|\div_C D|^2\,\mathcal{G}_0\,\Psi_N^2\,dX\bigg)^{\frac12}.
\\ \nonumber
&\lesssim\,C_D\,
\bigg(\iint_{\Omega}|\nabla u|^2\,\mathcal{G}_0\,\Psi_N^2\,dX\bigg)^{\frac12}\sigma(Q_0)^{\frac12},
\end{align}
where we have used $\|u\|_{L^{\infty}(\Omega)}=1$ and the last estimate is obtained as follows:
\begin{multline*}
\iint_{\Omega}|\div_C D |^2\,\mathcal{G}_0\,\Psi_N^2\,dX
\lesssim
\sum_{I\in\mathcal{W}_N}\mathcal{G}_0(X_I)\iint_{I^{**}}|\div_C D|^2\,dX
\\
\lesssim
\sum_{I\in\mathcal{W}_N}\ell(I)\iint_{I^{**}}|\div_C D|^2\,dX
\lesssim
\iint_{B_{Q_0}^*\cap\Omega}|\div_C D (X)|^2\delta(X)\,dX\
\lesssim C_D\,\sigma(Q_0),
\end{multline*}
where we have used Harnack's inequality, \eqref{G-delta}, the fact that the family $\{I^{**}\}_{I\in\mathcal{W}}$ has bounded overlap, \eqref{definicionkappa12}, and the last estimate follows from \eqref{carleson-B0-gral}, the fact that $r(B_{Q_0}^*)=2\kappa_0r_{Q_0}\le 2\kappa_0\ell(Q_0)\le 2\kappa_0\diam(\partial\Omega)/M_0<\diam(\partial\Omega)$ by our choice of $M_0$, and the Ahlfors regularity of $\partial\Omega$.

At this point we can collect \eqref{firstbound}--\eqref{acot-II2} and use Young's inequality to conclude that
\begin{multline*}
\iint_{\Omega}|\nabla u|^2\,\mathcal{G}_0\,\Psi_N^2\,dX
\le
C\sigma(Q_0)
+
C\bigg(\iint_{\Omega}|\nabla u|^2\,\mathcal{G}_0\,\Psi_N^2\,dX\bigg)^{\frac12}\sigma(Q_0)^{\frac12}
\\
\le
\frac{C(2+C)}2\sigma(Q_0)
+
\frac12 \iint_{\Omega}|\nabla u|^2\,\mathcal{G}_0\,\Psi_N^2\,dX.
\end{multline*}
The last term is finite since $\supp(\Psi_N)\subset \overline{\Omega^{*}_{\mathcal{F}_N,Q_0}}$ which is a compact subset of $\Omega$,
$u\in W^{1,2}_{\rm loc}(\Omega)$, $\mathcal{G}_0\in L^\infty_{\rm loc}(\Omega\setminus\{X_0\})$, \eqref{eq:X0-TQ}, and \eqref{definicionkappa12}. Hence we can hide it and use Lemma \ref{adaptedcutoff} to conclude as desired that
\[
\iint_{\Omega_{\mathcal{F}_N,Q_0}}|\nabla u|^2\,\mathcal{G}_0\,dX
\lesssim
\iint_{\Omega}|\nabla u|^2\,\mathcal{G}_0\,\Psi_N^2\,dX
\lesssim
\sigma(Q_0).
\]
This completes the proof, see \eqref{conc-pdelemma}.
\end{proof}


\begin{thebibliography}{AHM{\etalchar{+}}2}
	
	\bibitem[AHLT]{MR1879847}
	Pascal Auscher, Steve Hofmann, John~L. Lewis, and Philippe Tchamitchian.
	\newblock Extrapolation of {C}arleson measures and the analyticity of {K}ato's
	square-root operators.
	\newblock {\em Acta Math.}, {\bf 187}(2):161--190, 2001.
	
	\bibitem[AHMTT]{MR1934198}
	Pascal Auscher, Steve Hofmann, Camil Muscalu, Terence Tao, and Christoph
	Thiele.
	\newblock Carleson measures, trees, extrapolation, and {$T(b)$} theorems.
	\newblock {\em Publ. Mat.}, {\bf 46}(2):257--325, 2002.
	
	\bibitem[AHMNT]{MR3626548}
	Jonas Azzam, Steve Hofmann, Jos{\'e}~M. Martell, Kaj Nystr{\"o}m, and Tatiana
	Toro.
	\newblock A new characterization of chord-arc domains.
	\newblock {\em J. Eur. Math. Soc}, {\bf 19}(4), 2014.
	
	\bibitem[BJ]{BJ}
	Christopher~J. Bishop and Peter~W. Jones.
	\newblock Harmonic measure and arclength.
	\newblock {\em Ann. of Math.}, {\bf 132}(3):511--547, 1990.
	
	\bibitem[Bou]{bou}
	Jean Bourgain.
	\newblock On the {H}ausdorff dimension of harmonic measure in higher dimension.
	\newblock {\em Invent. Math.}, {\bf 87}(3):477--483, 1987.
	
	\bibitem[Car]{Car}
	Lennart Carleson.
	\newblock Interpolations by bounded analytic functions and the corona problem.
	\newblock {\em Ann. of Math.}, {\bf 76}(3):547--559, 1962.
	
	\bibitem[CF]{coifman1974}
	Ronald~R. Coifman and Charles~L. Fefferman.
	\newblock Weighted norm inequalities for maximal functions and singular
	integrals.
	\newblock {\em Stud. Math.}, {\bf 51}(3):241--250, 1974.
	
	\bibitem[CG]{CG}
	Lennart Carleson and John Garnett.
	\newblock Interpolating sequences and separation properties.
	\newblock {\em J. Analyse Math.}, {\bf 28}(1):273--299, 1975.
	
	\bibitem[CHM]{cavhofmartell}
	Juan Cavero, Steve Hofmann, and Jos\'e~M. Martell.
	\newblock Perturbations of elliptic operators in 1-sided chord-arc domains.
	{P}art {I}: Small and large perturbation for symmetric operators.
	\newblock {\em Trans. Amer. Math. Soc.}, {\bf 371}(4):2797--2835, 2019.
	
	\bibitem[Chr]{christ}
	Michael Christ.
	\newblock A {$T(b)$} theorem with remarks on analytic capacity and the {C}auchy
	integral.
	\newblock {\em Colloq. Math.}, {\bf 60/61}(2):601--628, 1990.
	
	\bibitem[Dah]{MR0466593}
	Bj{\"o}rn E.~J. Dahlberg.
	\newblock Estimates of harmonic measure.
	\newblock {\em Arch. Rational Mech. Anal.}, {\bf 65}(3):275--288, 1977.
	
	\bibitem[DJ]{MR1078740}
	Guy David and David~S. Jerison.
	\newblock Lipschitz approximation to hypersurfaces, harmonic measure, and
	singular integrals.
	\newblock {\em Indiana Univ. Math. J.}, {\bf 39}(3):831--845, 1990.
	
	\bibitem[DS1]{davidsemmes1}
	Guy David and Stephen Semmes.
	\newblock Singular integrals and rectifiable sets in {${\bf R}^n$}: {B}eyond
	{L}ipschitz graphs.
	\newblock {\em Ast\'erisque}, {\bf 193}:152, 1991.
	
	\bibitem[DS2]{davidsemmes2}
	Guy David and Stephen Semmes.
	\newblock {\em Analysis of and on uniformly rectifiable sets}, volume~38 of
	{\em Mathematical Surveys and Monographs}.
	\newblock American Mathematical Society, Providence, RI, 1993.
	
	\bibitem[FKP]{MR1114608}
	Robert~A. Fefferman, Carlos~E. Kenig, and Jill Pipher.
	\newblock The theory of weights and the {D}irichlet problem for elliptic
	equations.
	\newblock {\em Ann. of Math. (2)}, {\bf 134}(1):65--124, 1991.
	
	\bibitem[GR]{MR807149}
	Jos\'e {Garc{\'{\i}}a-Cuerva} and Jos\'e~L. {Rubio de Francia}.
	\newblock {\em Weighted norm inequalities and related topics}, volume 116 of
	{\em North-Holland Mathematics Studies}.
	\newblock North-Holland Publishing Co., Amsterdam, 1985.
	\newblock Mathematical Notes, 104.
	
	\bibitem[HKM]{HKM}
	Juha Heinonen, Tero Kilpel\"ainen, and Olli Martio.
	\newblock {\em Nonlinear potential theory of degenerate elliptic equations}.
	\newblock Oxford University Press, New York, 1993.
	
	\bibitem[HL]{MR1828387}
	Steve Hofmann and John~L. Lewis.
	\newblock {\em The Dirichlet Problem for Parabolic Operators with Singular
		Drift Terms}.
	\newblock Number 719 in American Mathematical Society: Memoirs of the American
	Mathematical Society. American Mathematical Society, 2001.
	
	\bibitem[HLMN]{hoflemartellnystrom}
	Steve Hofmann, Phi Le, Jos{\'e}~M. Martell, and Kaj Nystr{\"o}m.
	\newblock The weak-${A}_\infty$ property of harmonic and $p$-harmonic measures
	implies uniform rectifiability.
	\newblock {\em Anal. PDE}, {\bf 10}(3):513--558, 2017.
	
	\bibitem[HM1]{MR2655385}
	Steve Hofmann and Jos\'e~M. Martell.
	\newblock A note on {$A_\infty$} estimates via extrapolation of {C}arleson
	measures.
	\newblock In {\em The {AMSI}-{ANU} {W}orkshop on {S}pectral {T}heory and
		{H}armonic {A}nalysis}, volume~44 of {\em Proc. Centre Math. Appl. Austral.
		Nat. Univ.}, pages 143--166. Austral. Nat. Univ., Canberra, 2010.
	
	\bibitem[HM2]{MR2833577}
	Steve Hofmann and Jos\'e~M. Martell.
	\newblock {$A_\infty$} estimates via extrapolation of {C}arleson measures and
	applications to divergence form elliptic operators.
	\newblock {\em Trans. Amer. Math. Soc.}, {\bf 364}(1):65--101, 2012.
	
	\bibitem[HM3]{hofmartell}
	Steve Hofmann and Jos\'e~M. Martell.
	\newblock Uniform rectifiability and harmonic measure {I}: Uniform
	rectifiability implies {P}oisson kernels in ${L}^p$.
	\newblock {\em Ann. Sci. École Norm. Sup.}, {\bf 47}(3):577--654, 2014.
	
	\bibitem[HMMM]{HMMM}
	Steve Hofmann, Dorina Mitrea, Marius Mitrea, and Andrew~J. Morris.
	\newblock {\em ${L}^p$-Square Function Estimates on Spaces of Homogeneous Type
		and on Uniformly Rectifiable Sets}.
	\newblock Number 1159 in Memoirs of the American Mathematical Society. American
	Mathematical Society, 2017.
	
	\bibitem[HMT1]{HMT-general}
	Steve Hofmann, Jos\'e~M. Martell, and Tatiana Toro.
	\newblock General operators on 1-sided {NTA} domains.
	\newblock {\em work in progress}, 2014.
	
	\bibitem[HMT2]{HMT-var}
	Steve Hofmann, Jos\'e~M. Martell, and Tatiana Toro.
	\newblock ${A}_\infty$ implies {NTA} for variable coefficients.
	\newblock {\em J. Differential Equations}, {\bf 263}(10):6147--6188, 2017.
	
	\bibitem[HMUT]{hofmartelltuero}
	Steve Hofmann, Jos{\'e}~M. Martell, and Ignacio Uriarte-Tuero.
	\newblock Uniform rectifiability and harmonic measure, {II}: {P}oisson kernels
	in {$L^p$} imply uniform rectifiability.
	\newblock {\em Duke Math. J.}, {\bf 163}(8):1601--1654, 2014.
	
	\bibitem[JK]{MR676988}
	David~S. Jerison and Carlos~E. Kenig.
	\newblock Boundary behavior of harmonic functions in nontangentially accessible
	domains.
	\newblock {\em Adv. in Math.}, {\bf 46}(1):80--147, 1982.
	
	\bibitem[KKPT]{KKiPT}
	Carlos~E. Kenig, Bernd Kirchheim, Jill Pipher, and Tatiana Toro.
	\newblock Square functions and the {$A_\infty$} property of elliptic measures.
	\newblock {\em J. Geom. Anal.}, {\bf 26}(3):2383--2410, 2014.
	
	\bibitem[Lav]{Lav}
	Mikhail~A. Lavrentiev.
	\newblock Boundary problems in the theory of univalent functions ({R}ussian).
	\newblock {\em Math Sb.}, {\bf 43}:815--846, 1936.
	
	\bibitem[LM]{MR1020043}
	John~L. Lewis and Margaret A.~M. Murray.
	\newblock Regularity properties of commutators and layer potentials associated
	to the heat equation.
	\newblock {\em Trans. Amer. Math. Soc.}, {\bf 328}(2):815--842, 1991.
	
	\bibitem[MPT1]{MR3107693}
	Emmanouil Milakis, Jill Pipher, and Tatiana Toro.
	\newblock Harmonic analysis on chord arc domains.
	\newblock {\em J. Geom. Anal.}, {\bf 23}(4):2091--2157, 2013.
	
	\bibitem[MPT2]{MPT2}
	Emmanouil Milakis, Jill Pipher, and Tatiana Toro.
	\newblock Perturbations of elliptic operators in chord arc domains.
	\newblock {\em Harmonic analysis and partial differential equations, Contemp.
		Math.}, {\bf 612}:143--161, 2014.
	
	\bibitem[RR]{RR}
	Frigyes Riesz and Marcel Riesz.
	\newblock Uber die randwerte einer analytischen funktion.
	\newblock {\em Comptes Rendus du Quatri\`eme Congr\`es des Math\'ematiciens
		Scandinaves}, pages 27--44, 1916.
	
	\bibitem[Sem]{Sem}
	Stephen Semmes.
	\newblock Analysis vs. geometry on a class of rectifiable hypersurfaces in
	{$\mathbb{R}^n$}.
	\newblock {\em Indiana University Mathematics Journal}, {\bf 39}(4):1005--1035,
	1990.
	
	\bibitem[Zha]{ZHAO}
	Zihui Zhao.
	\newblock {BMO} solvability and ${A}_\infty$ condition of the elliptic measures
	in uniform domains.
	\newblock {\em J. Geom. Anal.}, {\bf 28}(2):866--908, 2018.
	

	
\end{thebibliography}

\newcommand{\etalchar}[1]{$^{#1}$}

\end{document}